\documentclass{elsarticle}

\usepackage[english]{babel}
\usepackage{dsfont} 
\usepackage{graphicx}
\usepackage{color}
\usepackage{hyperref}
\usepackage{enumerate}
\usepackage{amssymb,empheq} 
\usepackage{amsthm} 
\usepackage{cleveref}
\usepackage{xcolor}

\allowdisplaybreaks
\newtheorem{thm}{Theorem}[subsection]
\newtheorem{boasthm}[thm]{Boas' Theorem}
\newtheorem{havithm}[thm]{Haviland's Theorem}
\newtheorem{hausthm}[thm]{Hausdorff's Theorem}
\newtheorem{hamthm}[thm]{Hamburger's Theorem}
\newtheorem{karthm}[thm]{Karlin's Theorem}
\newtheorem{stielthm}[thm]{Stieltjes' Theorem}
\newtheorem{svethm}[thm]{{\v{S}}venco's Theorem}
\newtheorem{richthm}[thm]{Richter's Theorem}
\newtheorem{lem}[thm]{Lemma}

\newtheorem{cor}[thm]{Corollary}
\newtheorem{karcor}[thm]{Karlin's Corollary}

\theoremstyle{definition}

\newtheorem{dfn}[thm]{Definition}
\newtheorem{exm}[thm]{Example}
\newtheorem{exms}[thm]{Examples}

\theoremstyle{remark}
\newtheorem{rem}[thm]{Remark}

\newcommand{\exmsymbol}{\hfill$\circ$}

\newcommand{\nset}{\mathds{N}}

\newcommand{\rset}{\mathds{R}}

\newcommand{\diff}{\mathrm{d}}

\newcommand{\lin}{\mathrm{lin}\,}

\newcommand{\pos}{\mathrm{Pos}}

\newcommand{\sign}{\mathrm{sgn}}

\newcommand{\inter}{\mathrm{int}\,}

\newcommand{\cF}{\mathcal{F}}
\newcommand{\cG}{\mathcal{G}}
\newcommand{\cH}{\mathcal{H}}

\newcommand{\cX}{\mathcal{X}}

\newcommand{\cY}{\mathcal{Y}}
\newcommand{\cZ}{\mathcal{Z}}

\newcommand{\fA}{\mathfrak{A}}

\biboptions{sort&compress}

\author{Philipp J.\ di Dio}
\ead{philipp.didio@uni-konstanz.de}
\address{Department of Mathematics and Statistics, University of Konstanz, Universit\"atsstra{\ss}e 10, D-78464 Konstanz, Germany}
\address{Zukunftskolleg, Universtity of Konstanz, Universit\"atsstra{\ss}e 10, D-78464 Konstanz, Germany}
\address{eMail: philipp.didio@uni-konstanz.de}

\journal{arXiv}

\title{The Early History of Moment Problems and Non-Negative Polynomials with Gaps: Sparse Moment Problems, Sparse Positivstellensätze, and Sparse Nichtnegativstellensätze from a T-System Point of View}

\begin{document}

\begin{abstract}
We deal with and investigate sparse univariate Positivstellensätze, Nichtnegativstellensätze, and solutions to sparse moment problems.
The paper relies heavily on results on T-systems by Karlin in 1963 and by Karlin and Studden in 1966.
We gain complete descriptions of all sparse strictly positive and sparse non-negative algebraic polynomials on $[a,b]$ with $a\geq 0$ and on $[0,\infty)$. We extend, simplify, and solve the sparse Hausdorff and Stieltjes moment problem with these results and the methods of adapted spaces and T-systems.
\end{abstract}


\maketitle

\setcounter{tocdepth}{2}
\tableofcontents

\section{Introduction}


The theory of moments (or the moment problem) has been connected to non-negative polynomials for a long time and this connection is well-known since Haviland \cite{havila36} or even dates back further.

The classical moment problem is the following:
Given a closed set $K\subseteq\rset^n$ and a real sequence $s = (s_\alpha)_{\alpha\in I}$ with $I\subseteq\nset_0^n$. When does there exist a measure $\mu$ on $K$ such that
\[s_\alpha = \int_K x^\alpha~\diff\mu(x)\]
holds for all $\alpha\in I$?

For more on the theory of moments see e.g.\ \cite{shohat43,akhiezClassical,karlinStuddenTSystemsBook,kreinMarkovMomentProblem,
marshallPosPoly,lauren09,schmudMomentBook} and references therein.

In modern times the theory of moments and the theory of non-negative polynomials were revived by \cite{schmud91} and then put to useful applications, see e.g.\ \cite{lasserreSemiAlgOpt}.

By applications and specially by the need for efficient and fast algorithms the focus more and more turned the last years (and decades) to sparse systems, i.e., the index set $I\subsetneq\nset_0^n$ is not all $\nset_0^n$ and specially not all $\alpha$ with $|\alpha|\leq d$ for some $d\in\nset_0^n$ are present.
It should not be surprising that these sparse systems were studied theoretically.
More surprising is it that the early results in this field are not well-known or even completely forgotten.
And unfortunately it came recently to our attention that several known results are being reproved in weaker versions \cite{scheider23}.

The main purpose of this article is to review the early results in the theory of sparse moment problems and to show how important results and especially sparse Positivstellensätze, sparse Nichtnegativstellensätze, and sparse moment problems follow from these early results.
All results presented here are not contained in the modern literature about the theory of moments \cite{marshallPosPoly,schmudMomentBook}, about real algebraic geometry \cite{bochnak98}, or about (sparse) polynomial optimization \cite{lasserreSemiAlgOpt,margronSparsePolyOpt}.

We hope that this treatment will also be useful for the emerging works of moment problems and polynomials on curves since these often reduce to the univariate polynomial case \cite{zalar23}.

By the title we only look at early results (and their applications). By ``early'' we mean everything up to and including 1966.
Everything between 1967 and up to 1991 we consider ``modern'' and everything after ``contemporary''.
Modern and contemporary results are not considered here since they deserve more space than this article can give.
The year 1966 is chosen since in 1966 the extensive research monograph by Samuel Karlin and William J.\ Studden about T-systems appeared \cite{karlinStuddenTSystemsBook}.
This monograph is an extensive follow up of the work by Karlin in \cite{karlin63}.
Both works solve important problems in the theory of T-systems.
The theory of T-systems is the theoretical framework where e.g.\ sparse univariate algebraic polynomial systems were investigated in.
The year 1991 is chosen since then the first denominator free description of strictly positive polynomials appeared \cite{schmud91} reviving a large part in real algebraic geometry.

The article is structured as follows.
In the next \Cref{sec:prelim} we shortly introduce the notations in this article. 
In \Cref{sec:beginning} we shortly present the ``usual suspects'' (classical results without gaps) and the two first explicit studies of problems with gaps.
We will meet there also Richter's Theorem and Boas' Theorem.
In \Cref{sec:tsystems} we will introduce the theory of T-systems and show their basic properties with a special emphasis on zeros and non-negativity.
By far the most important part is \Cref{sec:sparsePosNiNeg} where we look at the results in \cite{karlin63,karlinStuddenTSystemsBook} and apply them to get sparse algebraic Positivstellensätze and Nichtnegativstellensätze.
Additionally, they are used to solve and even extend the early sparse moment problems from \Cref{sec:beginning}.
In \Cref{sec:summary} we sum up the results.

All results are presented with proofs as far as possible.
Several are collected from the literature but translated to nowadays mathematical language.
Also some missing steps are filled in and errors are corrected.

\section{Preliminaries}
\label{sec:prelim}

Let $n\in\nset$, $K\subseteq\rset^n$ be closed, and $s = (s_\alpha)_{\alpha\in\nset_0^n}$ be a real sequence. We say that $s$ is a $K$-moment sequence if there exists a measure $\mu$ on $K$ such that
\[s_\alpha = \int_K x^\alpha~\diff\mu(x).\]
The measure $\mu$ is called a representing measure. Unless otherwise denoted as signed measure all measures are positive. The moment sequence $s$ is called determined if $\mu$ is unique. We call $s$ a truncated moment sequence when only finitely many $s_\alpha$ with $\alpha\in\nset_0$ are known.

For a real sequence $s=(s_\alpha)_{\alpha\in\nset_0}$ we call the linear functional
$L_s:\rset[x_1,\dots,x_n]\to\rset$ defined by $L_s(x^\alpha) = s_\alpha$
the Riesz functional. For $\beta\in\nset_0$ we define $X^\beta s = (s_{\alpha+\beta})_{\alpha\in\nset_0^n}$ the shifted sequence.

For a sequence $s=(s_\alpha)_{\alpha\in\nset_0^n}$ we define the Hankel matrix $\cH(s) := (s_{\alpha,\beta})_{\alpha,\beta\in\nset_0^n}$.
For $K\subseteq\rset^n$ we set $\pos(K) := \{f\in\rset[x_1,\dots,x_n] \,|\, f\geq 0\ \text{on}\ K\}$.
For any set $\cX$ we denote by $|\cX|$ the cardinality of $\cX$.

\section{The Beginning of the Moment Problem}
\label{sec:beginning}

\subsection{The Usual Suspects: Well-known Classical Results without Gaps}

The first moment problem that was solved is the following.

\begin{stielthm}[{\cite{stielt94}}]\label{thm:stieltjesMP}
Let $s = (s_i)_{i\in\nset_0}$ be a real sequence. The following are equivalent:
\begin{enumerate}[(i)]
\item $s$ is a $[0,\infty)$-moment sequence (Stieltjes moment sequence).

\item $L_s(p)\geq 0$ for all $p\in\pos([0,\infty))$.

\item $L_s(p^2)\geq 0$ and $L_{Xs}(p^2) = L_s(x\cdot p^2)\geq 0$ for all $p\in\rset[x]$.

\item $s$ and $Xs = (s_{i+1})_{i\in\nset_0}$ are positive semidefinite.

\item $\cH(s)\succeq 0$ and $\cH(Xs)\succeq 0$ for all $d\in\nset_0$.
\end{enumerate}
\end{stielthm}

\Cref{thm:stieltjesMP} in the original proof \cite{stielt94} does not use non-negative polynomials. Stieljes uses continued fractions and introduces new sequences which we (nowadays) denote by $s$ and $Xs$.

Stieltjes only proves (i) $\Leftrightarrow$ (iv). The implication (i) $\Leftrightarrow$ (ii) is \Cref{thm:haviland}, (ii) $\Leftrightarrow$ (iii) is the description of $\pos([0,\infty))$, and (iv) $\Leftrightarrow$ (v) is a reformulation of $s$ and $Xs$ being positive semi-definite.

The next moment problem that was solved is the following.

\begin{hamthm}[{\cite[Satz X and Existenstheorem (§8, p.\ 289)]{hamburger20}}]\label{thm:hamburgerMP}
Let $s = (s_i)_{i\in\nset_0}$ be a real sequence. The following are equivalent:
\begin{enumerate}[(i)]
\item $s$ is a $\rset$-moment sequence (Hamburger moment sequence or short moment sequence).

\item $L_s(p)\geq 0$ for all $p\in\pos(\rset)$.

\item $L_s(p^2)\geq 0$ for all $p\in\rset[x]$.

\item $s$ is positive semidefinite.

\item $\cH(s)\succeq 0$.
\end{enumerate}
\end{hamthm}

Hamburger proves similar to Stieltjes \cite{stielt94} the equivalence (i) $\Leftrightarrow$ (iv) via continued fractions.
In \cite[Satz XIII]{hamburger20} Hamburger solves the full moment problem by approximation with truncated moment problems.
This was later in a slightly more general framework reproved in \cite{stochel01}.
Hamburger needed to assume that the sequence of measures $\mu_k$ (which he called ``Belegungen'' and denoted by $\diff\Phi^{(k)}(u)$) to converge to some measure $\mu$ (condition 2 of \cite[Satz XIII]{hamburger20}).
Hamburgers additional condition 2 is nowadays replaced by the vague convergence and the fact that the solution set of representing measures is vaguely compact \cite[Thm.\ 1.19]{schmudMomentBook}, i.e., it assures the existence of a $\mu$ as required by Hamburger in the additional condition 2.

Shortly after Hamburger the moment problem on $[0,1]$ was solved.

\begin{hausthm}[{\cite[Satz II and III]{hausdo21}}]\label{thm:hausdorffMP}
Let $s = (s_i)_{i\in\nset_0}$ be a real sequence. The following are equivalent:
\begin{enumerate}[(i)]
\item $s$ is a $[0,1]$-moment sequence (Hausdorff moment sequence).

\item $L_s(p)\geq 0$ for all $p\in\pos([0,1])$.

\item $L_s(p^2)\geq 0$, $L_{Xs}(p^2)\geq 0$, and $L_{(1-X)s}(p^2)\geq 0$ for all $p\in\rset[x]$.

\item $s$, $Xs$, and $(1-X)s$ are positive semidefinite.

\item $\cH(s)\succeq 0$, $\cH(Xs)\succeq 0$, and $\cH((1-X)s)\succeq 0$.
\end{enumerate}
\end{hausthm}

Hausdorff proves the equivalence (i) $\Leftrightarrow$ (iii) via so called C-sequences.
In \cite{toeplitz11} Toeplitz treats general linear averaging methods.
In \cite{hausdo21} Hausdorff uses these.
Let the infinite dimensional matrix $\lambda = (\lambda_{i,j})_{i,j\in\nset_0}$ be row-finite, i.e., for every row $i$ only finitely many $\lambda_{i,j}$ are non-zero.
Then the averaging method
\[A_i = \sum_{j\in\nset_0} \lambda_{i,j} a_j\]
shall be consistent: If $a_j \to \alpha$ converges then $A_i\to\alpha$ converges to the same limit.
Toeplitz proved a necessary and sufficient condition on $\lambda$ for this property.
Hausdorff uses only part of this property. He calls a matrix $(\lambda_{i,j})_{i,j\in\nset_0}$ with the property that a convergent sequence $(a_j)_{j\in\nset_0}$ is mapped to a convergent sequence $(A_j)_{j\in\nset_0}$ (the limit does not need to be preserved) a C-matrix (convergence preserving matrix).
Hausdorff gives the characterization of C-matrices \cite[p.\ 75, conditions (A) -- (C)]{hausdo21}.
Additionally, if $\lambda$ is a C-matrix and a diagonal matrix with diagonal entries $\lambda_{i,i} = s_i$ then $s = (s_i)_{i\in\nset_0}$ is called a C-sequence.
The equivalence (i) $\Leftrightarrow$ (iii) is then shown by Hausdorff in the result that a sequence is a $[0,1]$-moment sequence if and only it is a C-sequence \cite[p.\ 102]{hausdo21}.

A much simpler approach to solve the $K$-moment problem for any closed $K\subseteq\rset^n$, $n\in\nset$, was presented by Haviland.
He no longer used continued fractions but employed the Riesz(--Markov--Kakutani) representation theorem, i.e., representing a linear functional by integration.

The present Riesz--Markov--Kakutani representation theorem was developed in several stages.
A first version for continuous functions on the unit interval $[0,1]$ is by F.\ Riesz \cite{riesz09}.
It was extended by Markov to some non-compact spaces \cite{markov38} and then by Kakutani to locally compact Hausdorff spaces \cite{kakutani41}.
Interestingly, it already follows from Daniell's Representation Theorem \cite{daniell18,daniell20} with Urysohn's Lemma \cite{urysohn25}.

Haviland proved the following.

\begin{havithm}[{\cite[Theorem]{havila36}}, see also {\cite[Theorem]{havila35}} for $K=\rset^n$]\label{thm:haviland}
Let $n\in\nset$, $K\subseteq\rset^n$ be closed, and $s = (s_\alpha)_{\alpha\in\nset_0^n}$ be a real sequence. The following are equivalent:
\begin{enumerate}[(i)]
\item $s$ is a $K$-moment sequence.

\item $L_s(p)\geq 0$ for all $p\in\pos(K)$.
\end{enumerate}
\end{havithm}

In \cite[Theorem]{havila35} Haviland proves ``only'' the case $K=\rset^n$ with the extension method by M.\ Riesz.
In \cite[Theorem]{havila36} this is extended to any closed $K\subseteq\rset^n$.
The idea to do so is attributed by Haviland to Aurel Wintner \cite[p.\ 164]{havila36}:
``A.\ Wintner has subsequently suggested that it should be possible to extend this result [\cite[Theorem]{havila35}] by requiring that the distribution function [measure] solving the problem have a spectrum [support] contained in a preassigned set, a result which would show the well-known criteria for the various standard special momentum problems (Stieltjes, Herglotz [trigonometric], Hamburger, Hausdorff in one or more dimensions) to be put particular cases of the general $n$-dimensional momentum problem mentioned above.
The purpose of this note is to carry out this extension.''

In \cite{havila36} after the general Theorem \ref{thm:haviland} Haviland then goes through all the classical results (Theorems \ref{thm:stieltjesMP} to \ref{thm:hausdorffMP}, and the Herglotz (trigonometric) moment problem on the unit circle which we did not included here) and shows how all these results (i.e., conditions on the sequences) are recovered from the at this point known representations of non-negative polynomials.

For the Hamburger moment problem Haviland uses
\[\pos(\rset) = \left\{ f^2 + g^2 \,\middle|\, f,g\in\rset[x]\right\}\]
which was already known to Hilbert \cite{hilbert88}.
For the Stieltjes moment problem he uses
\begin{equation}\label{eq:pos0infty1}
\pos([0,\infty)) = \left\{ f_1^2 + f_2^2 + x\cdot (g_1^2 + g_2^2) \,\middle|\, f_1,f_2,g_1,g_2\in\rset[x]\right\}
\end{equation}
with the reference to P\'olya and Szeg{\"o} (previous editions of \cite{polya64,polya70}).
In \cite[p.\ 82, ex.\ 45]{polya64} the representation (\ref{eq:pos0infty1}) is still included while it was already known before, see \cite[p.\ 6, footnote]{shohat43}, that
\begin{equation}\label{eq:pos0infty2}
\pos([0,\infty)) = \left\{ f^2 + x\cdot g^2 \,\middle|\, f,g\in\rset[x]\right\}
\end{equation}
is sufficient.
Also in \cite[Prop.\ 3.2]{schmudMomentBook} the representation (\ref{eq:pos0infty1}) is used, not the representation (\ref{eq:pos0infty2}).

For the $[-1,1]$-moment problem Haviland uses
\[\pos([-1,1]) = \left\{ f^2 + (1-x^2)\cdot g^2 \,\middle|\, f,g\in\rset[x]\right\}.\]
For the Hausdorff moment problem he uses that any non-negative polynomial on $[0,1]$ is a linear combination of $x^m\cdot (1-x)^{p-m}$, $m,p\in\nset_0$, $p\geq m$, with non-negative coefficients.
For the two-dimensional Hausdorff moment problem he uses that any non-negative polynomial on $[0,1]^2$ is a linear combination of $x^m\cdot y^n\cdot (1-x)^{p-m}\cdot (1-y)^{q-n}$, $n,m,q,p\in\nset_0$, $p\geq m$, $q\geq n$, with non-negative coefficients \cite{hildeb33}.
Hildebrandt and Schoenberg \cite{hildeb33} already solved the moment problem on $[0,1]^2$ (and more generally on $[0,1]^n$ for all $n\in\nset$) getting the same result as Haviland.
The idea of using $\pos(K)$-descriptions to solve the moment problem was therefore already used by Hildebrandt and Schoenberg in 1933 \cite{hildeb33} before Haviland uses this in \cite{havila35} and generalized this in \cite{havila36} as suggested to him by Wintner.

With these broader historical remarks we see that of course more people are connected to Theorem \ref{thm:haviland}.
It might also be appropriate to call Theorem \ref{thm:haviland} the \emph{Haviland--Wintner} or \emph{Haviland--Hildebrand--Schoenberg--Wintner Theorem}.
But as so often, the list of contributors is long (and maybe even longer) and hence the main contribution (the general proof) is rewarded by calling it just Haviland Theorem.

As one other solved moment problem of the long list (our list here is far from complete) is the following.

\begin{svethm}[{\cite{svenco39}}]
Let $s=(s_i)_{i\in\nset_0}$ be a real sequence. The following are equivalent:
\begin{enumerate}[(i)]
\item $s$ is a $(-\infty,0]\cup [1,\infty)$-moment sequence.

\item $L_s(p)\geq 0$ for all $p\in\pos((-\infty,0]\cup[1,\infty))$.

\item $L_s(p^2)\geq 0$, $L_{(X^2-X)s}(p^2)\geq 0$ for all $p\in\rset[x]$.

\item $s$ and $(X^2-X)s$ are positive semi-definite.

\item $\cH(s)\succeq 0$ and $\cH((X^2-X)s)\succeq 0$.
\end{enumerate}
\end{svethm}

All moment problems on closed and semi-algebraic sets $K\subseteq\rset$ follow nowadays easily from \Cref{thm:haviland} and the fact that any preodering from a natural description of $K$ is saturated, see e.g.\ \cite[Prop.\ 2.7.3]{marshallPosPoly}.

The higher dimensional moment problem is much harder than the one-dimensional moment problem and in general it is not solved.
The reason is that a description of $\pos(K)$ is in general unknown.
A huge progress in this field was done by Konrad Schmüdgen in 1991 \cite{schmud91} where he solved the $K$-moment problem for compact semi-algebraic sets $K\subset\rset^n$, $n\geq 2$.
As a corollary he gained a complete description of strictly positive $f\in\pos(K)$.
These and subsequence results are discussed elsewhere \cite{marshallPosPoly,lauren09}.

\subsection{Early Results with Gaps}

The early history of moment problems with gaps is very thin. We discuss only \cite{hausdo21a} and \cite{boas39}.

Hausdorff just solved \Cref{thm:hausdorffMP} in \cite{hausdo21} (submitted 11th February 1920) and in \cite{hausdo21a} (submitted 8th September 1920) he treats
\[s_n = \int_0^1 x^{k_n}~\diff\mu(x)\]
with
\[k_0 = 0 < k_1 < k_2 < \dots < k_n < \dots \]
for a sequence of real numbers, i.e., not necessarily in $\nset_0$.
See also \cite[p.\ 104]{shohat43}.
Since Hausdorff in \cite{hausdo21a} did not have access to \Cref{thm:haviland} \cite{havila36} or the description of all non-negative linear combinations of $1,x^{k_1}, \dots, x^{k_n},\dots$ the results in \cite{hausdo21a} need complicated formulations and are not very strong.
Only with the description of non-negative linear combinations by Karlin \cite{karlin63} an easy formulation of the result is possible.
We will therefore postpone the exact formulation to \Cref{thm:sparseTruncHausd}, \ref{thm:sparseHausd}, and \ref{thm:generalSparseHausd} where we present easy proofs using also the theory of adapted spaces \cite{choquet69,phelpsLectChoquetTheorem,schmudMomentBook}.

In \cite{boas39} Boas investigates the Stieltjes moment problem ($K=[0,\infty)$) with gaps.
Similar to \cite{hausdo21a} the results are difficult to read and they are unfortunately incomplete since Boas (like Hausdorff) did not have access to the description of all non-negative or strictly positive polynomials with gaps (or more general exponents).
We will give the complete solution of the $[0,\infty)$-moment problem with gaps and more general exponents in \Cref{thm:sparseStieltjesMP}.

\subsection{Finitely Atomic Representing Measures: The Richter Theorem}

When working with a truncated moment sequence it is often useful in theory and applications to replace a representing measure with a finitely atomic measure without changing the moments.
That this is always possible for truncated moment sequences was first proved in full generality by Richter \cite{richte57}.

\begin{richthm}[{\cite[Satz 4]{richte57}}]\label{thm:richter}
Let $n\in\nset$, $(\cX,\fA)$ be a measurable space, and $\{f_i\}_{i=1}^n$ be a family of real measurable functions $f_i:\cX\to\rset$. Then for every measure $\mu$ on $\cX$ such that all $f_i$ are $\mu$-integrable, i.e.,
\[s_i := \int_\cX f_i(x)~\diff\mu(x) \quad <\infty\]
for all $i=1,\dots,n$, there exists a $K\in\nset$ with $K\leq n$, points $x_1,\dots,x_K\in\cX$ pairwise different, and $c_1,\dots,c_K\in (0,\infty)$ such that
\[s_i = \sum_{j=1}^K c_j\cdot f_i(x_j)\]
holds for all $i=1,\dots,n$.
\end{richthm}

The history of this result is often still misrepresented in the literature, even after K.\ Schmüdgen and the present author compared the different contributions and publication dates in detail in \cite{didioCone22}. With these historical remarks it also is appropriate to call Theorem \ref{thm:richter} the \emph{Richter--Rogosinski--Rosenbloom Theorem} \cite{richte57,rogosi58,rosenb52}.
Every other result before or after \cite{richte57} is only a special case and can easily be recovered from \Cref{thm:richter}, especially \cite{bayer06}.

Since \Cref{thm:richter} only needs a family of finitely many measurable functions it also includes all cases of gaps in the truncated moment theory.

\subsection{Signed Representing Measures: Boas' Theorem}

In the theory of moments almost exclusively the representation by non-negative measures is treated.
The reason is the following.

\begin{boasthm}[\cite{boas39a} or e.g.\ {\cite[p.\ 103, Thm.\ 3.11]{shohat43}}]\label{thm:boas}
Let $s = (s_i)_{i\in\nset_0}$ be a real sequence. Then there exist infinitely many signed measures $\mu$ on $\rset$ and infinitely many signed measures $\nu$ on $[0,\infty)$ such that
\[s_i = \int_\rset x^{i}~\diff\mu(x) = \int_0^\infty x^{i}~\diff\nu(x)\]
holds for all $i\in\nset_0$.
\end{boasthm}

\Cref{thm:boas} also holds in the $n$-dimensional case on $\rset^n$ and $[0,\infty)^n$ for any $n\in\nset$.
See \cite[p.\ 104]{shohat43} for an extension which kinds of measures can be chosen.

\Cref{thm:boas} also covers the case with gaps.
If any gaps in the real sequence $s$ are present then fill them with any real number you like.

\section{T-Systems}
\label{sec:tsystems}

We have seen the early attempts to deal with gaps in the moment problems.
A sufficient solution was at these times not possible.
Only with the introduction of so called T-systems and their rigorous investigation significant progress and finally complete solutions were possible.
For more on the early development and history of T-systems see \cite{krein51,karlin63,karlinStuddenTSystemsBook,kreinMarkovMomentProblem}.

\subsection{Definition and Basic Properties of T-Systems}

\begin{dfn}\label{dfn:tSystem}
Let $n\in\nset$, $\cX$ be a set with $|\cX|\geq n+1$, and let $\cF=\{f_i\}_{i=0}^n$ be a family of real functions $f_i:\cX\to\rset$.
We call any linear combination
\[f = \sum_{i=0}^n a_i\cdot f_i \quad\in\lin\cF\]
with $a_1,\dots, a_n\in\rset$ a \emph{polynomial}.
The family $\cF$ on $\cX$ is called a \emph{Tchebycheff system} (\emph{T-system}) \emph{of order $n$ on $\cX$} if any polynomial $f\in\lin\cF$ with $\sum_{i=0}^n a_i^2 > 0$ has at most $n$ zeros on $\cX$.

If $\cX$ is a topological space and $\cF$ is a family of continuous functions then we call $\cF$ a \emph{continuous T-system}.
If additionally $\cX$ is the unit circle then we call $\cF$ a \emph{periodic T-system}.
\end{dfn}

\begin{cor}\label{cor:restriction}
Let $n\in\nset$ and $\cF=\{f_i\}_{i=0}^n$ be a T-system of order $n$ on some $\cX$ with $|\cX|\geq n+1$. Let $\cY\subset\cX$ with $|\cY|\geq n+1$. Then $\cG := \{f_i|_\cY\}_{i=0}^n$ is a T-system of order $n$ on $\cY$.
\end{cor}
\begin{proof}
Let $f\in\lin\cF$. Then $f$ has at most $n$ zeros in $\cX$ and hence $f|_\cY$ has at most $n$ zeros in $\cY\subset\cX$.
Since for any $g\in\lin\cG$ there is a $f\in\lin\cF$ such that $g = f|_\cY$ we have the assertion.
\end{proof}

The set $\cX$ does not require any structure or property except $|\cX|\geq n+1$.

In the theory of T-systems we often deal with one special matrix.
We use the following abbreviation.

\begin{dfn}\label{dfn:kreinMatrix}
Let $n\in\nset$, $\{f_i\}_{i=0}^n$ be a family of real functions on a set $\cX$ with $|\cX|\geq n+1$. We define the matrix
\[\begin{pmatrix}
f_0 & f_1 & \dots & f_n\\ x_0 & x_1 & \dots & x_n
\end{pmatrix}
:=
\begin{pmatrix}
f_0(x_0) & f_1(x_0) & \dots & f_n(x_0)\\
f_0(x_1) & f_1(x_1) & \dots & f_n(x_1)\\
\vdots & \vdots & & \vdots\\
f_0(x_n) & f_1(x_n) & \dots & f_n(x_n)
\end{pmatrix} = (f_i(x_j))_{i,j=0}^n\]
for any $x_0,\dots,x_n\in\cX$.
\end{dfn}

\begin{lem}[see e.g.\ {\cite[p.\ 31]{kreinMarkovMomentProblem}}]\label{lem:determinant}
Let $n\in\nset$, $\cX$ be a set with $|\cX|\geq n+1$, and $\cF=\{f_i\}_{i=0}^n$ be a family of real functions $f_i:\cX\to\rset$. The following are equivalent:
\begin{enumerate}[(i)]
\item $\cF$ is a T-system of order $n$ on $\cX$.

\item The determinant
\[\det\begin{pmatrix}
f_0 & f_1 & \dots & f_n\\ x_0 & x_1 & \dots & x_n
\end{pmatrix}\]
does not vanish for any pairwise distinct points $x_0,\dots,x_n\in\cX$.
\end{enumerate}
\end{lem}
\begin{proof}
\underline{(i) $\Rightarrow$ (ii):}
Let $x_0,\dots,x_n\in\cX$ be pairwise distinct.
Since $\cF$ is a T-system we have that any non-trivial polynomial $f$ has at most $n$ zeros, i.e., the matrix
\[\begin{pmatrix}
f_0 & f_1 & \dots & f_n\\ x_0 & x_1 & \dots & x_n
\end{pmatrix}\]
has trivial kernel and hence its determinant is non-zero. Since $x_0,\dots,x_n\in\cX$ are arbitrary pairwise distinct we have (ii).

\underline{(ii) $\Rightarrow$ (i):} Assume there is a polynomial $f$ with $\sum_{i=0}^n a_i^2 >0$ which has the $n+1$ pairwise distinct zeros $z_0,\dots,z_n\in\cX$. Then the matrix
\[Z=\begin{pmatrix}
f_0 & f_1 & \dots & f_n\\ z_0 & z_1 & \dots & z_n
\end{pmatrix}\]
has non-trivial kernel since $0\neq (a_0,a_1,\dots,a_n)^T\in\ker Z$ and hence $\det Z = 0$ in contradiction to (ii).
\end{proof}

\begin{cor}\label{cor:uniqueDeter}
Let $n\in\nset$, and $\cF = \{f_i\}_{i=0}^n$ be a T-system of order $n$ on some $\cX\subseteq\rset$ with $|\cX|\geq n+1$. The following hold:
\begin{enumerate}[(i)]
\item The functions $f_0,\dots,f_n$ are linearly independent over $\cX$.

\item For any $f = \sum_{i=0}^n a_i\cdot f_i\in\lin\cF$ the coefficients $a_i$ are unique.
\end{enumerate}
\end{cor}
\begin{proof}
Follows immediately from \Cref{lem:determinant} (i) $\Rightarrow$ (ii).
\end{proof}

We even have the following.

\begin{thm}[see e.g.\ {\cite[p.\ 33]{kreinMarkovMomentProblem}}]
Let $n\in\nset$, $\cF$ be a T-system on some set $\cX$ with $|\cX|\geq n+1$, and let $x_0,\dots,x_n\in\cX$ be pairwise different points.
The following hold:
\begin{enumerate}[(i)]
\item Any $f\in\lin\cF$ is uniquely determined by its values $f(x_0),\dots, f(x_n)$.

\item For any $y_0,\dots,y_n\in\rset$ there exists a unique $f\in\lin\cF$ with $f(x_i)=y_i$ for all $i=0,\dots,n$.
\end{enumerate}
\end{thm}
\begin{proof}
\underline{(i):}
Since $f\in\lin\cF$ we have $f = \sum_{i=0}^n a_i\cdot f_i$.
Let $x_1,\dots,x_n\in\cF$ be pairwise distinct. Then by \Cref{lem:determinant} (i) $\Rightarrow$ (ii) we have that
\[\begin{pmatrix}
f(x_0)\\ \vdots\\ f(x_n)
\end{pmatrix} = \begin{pmatrix}
f_0 & f_1 & \dots & f_n\\ x_0 & x_1 & \dots & x_n
\end{pmatrix}\cdot \begin{pmatrix}
\alpha_0\\ \vdots\\ \alpha_n
\end{pmatrix} \]
has the unique solution $\alpha_0 = a_0$, \dots, $\alpha_n = a_n$.

\underline{(ii):}
By the same argument as in (i) the system
\[\begin{pmatrix}
y_0\\ \vdots\\ y_n
\end{pmatrix} = \begin{pmatrix}
f_0 & f_1 & \dots & f_n\\ x_0 & x_1 & \dots & x_n
\end{pmatrix}\cdot \begin{pmatrix}
\alpha_0\\ \vdots\\ \alpha_n
\end{pmatrix} \]
has the unique solution $\alpha_0 = a_0$, \dots, $\alpha_n = a_n$.
\end{proof}

So far we imposed no structure on $\cX$.
We now impose structure on $\cX$.
The following structural result was proved in \cite{mairhu56} for compact subsets $\cX$ of $\rset^n$ and for arbitrary compact sets $\cX$ in \cite{sieklu58,curtis59}.

\begin{thm}[{\cite[Thm.\ 2]{mairhu56}}, {\cite{sieklu58}}, {\cite[Thm.\ 8 and Cor.]{curtis59}}]\label{thm:msc}
Let $n\in\nset$ and $\cF$ be a continuous T-system of order $n$ on a topological space $\cX$. If $\cX$ is a compact metrizable space then $\cX$ can be homeomorphically embedded in the unit circle $\{(x,y)\in\rset^2 \,|\, x^2 + y^2 = 1\}$.
\end{thm}

\begin{cor}[{\cite[Thm.\ 8]{curtis59}}]
The order $n$ of a periodic T-system is even.
\end{cor}
\begin{proof}
Let $\varphi:[0,2\pi]\to S=\{(x,y)\in\rset^2 \,|\, x^2 + y^2\}$ with $\varphi(\alpha) = (\sin\alpha, \cos\alpha)$ and $\cF=\{f_i\}_{i=0}^n$ be a periodic T-system. Then the $f_i$ are continuous and hence also
\[\det\begin{pmatrix}
f_0 & f_1 & \dots & f_n\\ t_0 & t_1 & \dots & t_n
\end{pmatrix}\]
is continuous in $t_0,\dots,t_n\in S$. If $\cF$ is a T-system we have that
\[d(\alpha) := \det\begin{pmatrix}
f_0 & f_1 & \dots & f_n\\ \varphi(\alpha) & \varphi(\alpha+2\pi/(n+1)) & \dots & \varphi(\alpha + 2n\pi/(n+1))
\end{pmatrix}\]
in non-zero for all $\alpha\in [0,2\pi]$ and never changes singes. If $n$ is odd then $d(0) = -d(2\pi/(n+1))$ which is a contradiction. Hence, $n$ must be even.
\end{proof}

\subsection{Examples of T-systems}
\label{subsec:examples}

\begin{exms}[algebraic polynomials, see e.g.\ \cite{karlinStuddenTSystemsBook,kreinMarkovMomentProblem}]\label{exm:algPoly}
\begin{enumerate}[\bfseries (a)]
\item Let $n\in\nset$ and $\cX \subseteq \rset$ with $|\cX|\geq n+1$. Then the family $\cF = \{x^i\}_{i=0}^n$ is a T-system. This follows immediately from the Vandermonde determinant
\begin{equation}\label{eq:abuse}
\det\begin{pmatrix}
1 & x & \dots & x^n\\ x_0 & x_1 & \dots & x_n
\end{pmatrix} = \prod_{0\leq i<j\leq n} (x_j-x_i)
\end{equation}
for any $x_0,\dots,x_n\in\cX$.

Note that we abuse the notation for the algebraic polynomial cases.
The functions $f_0,\dots,f_n$ should not be denoted by $x^i$ but by
\[\cdot^{i}:\rset\to\rset,\ x\mapsto x^i.\]
However, then we have the notation
\[\begin{pmatrix}
\cdot^0 & \cdot^1 & \dots & \cdot^n\\ x_0 & x_1 & \dots & x_n
\end{pmatrix} \qquad\text{or more general}\qquad
\begin{pmatrix}
\cdot^{\alpha_0} & \cdot^{\alpha_1} & \dots & \cdot^{\alpha_n}\\
x_0 & x_1 & \dots & x_n
\end{pmatrix}\]
which seems hard to read.
For convenience we will therefore abuse the notation and use $x^i$ and (\ref{eq:abuse}).

\item Let $n\in\nset$, $\cX\subseteq [0,\infty)$ with $|\cX|\geq n+1$, and $\alpha_0 = 0 < \alpha_1 < \dots < \alpha_n$ be real numbers.
Then $\cF = \{x^{\alpha_i}\}_{i=0}^n$ is a T-system of order $n$ on $\cX$.

\item Let $n\in\nset$, $\cX\subseteq (0,\infty)$ with $|\cX|\geq n+1$, and $0 < \alpha_0 < \dots < \alpha_n$ be real numbers. Then $\cF = \{ x^{\alpha_i}\}_{i=0}^n$ is a T-system of order $n$ on $\cX$.\exmsymbol
\end{enumerate}
\end{exms}

\begin{exm}[see e.g.\ {\cite[p.\ 38]{kreinMarkovMomentProblem}}]
Let $n\in\nset$ and $\alpha_0 < \alpha_1 < \dots < \alpha_n$ be reals. Then
\[\cF = \left\{e^{\alpha_0 x}, e^{\alpha_1 x},\dots, e^{\alpha_n x}\right\}\]
is a T-system on any $\cX\subseteq\rset$ with $|\cX|\geq n+1$.\exmsymbol
\end{exm}

\begin{exm}[see e.g.\ {\cite[p.\ 37-38]{kreinMarkovMomentProblem}}]
Let $n\in\nset$ and $\alpha_0 < \alpha_1 < \dots < \alpha_n$ be reals. Then
\[\cF = \left\{\frac{1}{x+\alpha_0}, \frac{1}{x+\alpha_1}, \dots, \frac{1}{x+\alpha_n}\right\}\]
is a continuous T-system on any $[a,b]$ or $[a,\infty)$ with $-\alpha_0 < a < b$.\exmsymbol
\end{exm}

\begin{exm}[see e.g.\ {\cite[p.\ 38]{kreinMarkovMomentProblem}}]
Let $n\in\nset$ and let $f\in C^n(\cX)$ with $\cX=[a,b]$, $a<b$, and $f^{(n)}>0$ on $\cX$. Then
\[\cF = \{1,x,x^2,\dots, x^{n-1},f\}\]
is a continuous T-system of order $n$ on $\cX = [a,b]$. We can also allow $\cX = (a,b)$, $[a,\infty)$, $(-\infty,b)$, \dots .\exmsymbol
\end{exm}

\begin{exm}[see e.g.\ {\cite[p.\ 10]{karlinStuddenTSystemsBook}}]\label{exm:scaling}
Let $n\in\nset$ and $\cF = \{f_i\}_{i=0}^n$ be a (continuous) T-systems on $\cX\subseteq\rset$ with $|\cX|\geq n+1$.
Then for any (continuous) function $r:\cX\to (0,\infty)$ the family $\{r\cdot f_i\}_{i=0}^n$ is a (continuous) T-system.\exmsymbol
\end{exm}

\begin{exm}
Let $n\in\nset$, $\{f_i\}_{i=0}^n$ be a (continuous) T-system of order $n$ on $\cX\subseteq\rset$ and let $g:\cY\subseteq\rset\to\cX$ be a strictly increasing (continuous) function. 
Then $\{f_i\circ g\}_{i=0}^n$ is a (continuous) T-systems of order $n$ on $\cY$.
\exmsymbol
\end{exm}

\subsection{Non-Negativity, Zeros, and Determinantal Representations of Polynomials in T-Systems}

\begin{thm}[see e.g.\ {\cite[p.\ 20]{karlinStuddenTSystemsBook}} or {\cite[p.\ 33]{kreinMarkovMomentProblem}}]\label{thm:detRepr}
Let $n\in\nset$, $\cF = \{f_i\}_{i=0}^n$ be a T-system on some set $\cX$ with $|\cX|\geq n+1$, $x_1,\dots,x_n\in\cX$ be $n$ distinct points, and let $f\in\lin\cF$ be a polynomial. The following are equivalent:
\begin{enumerate}[(i)]
\item $f(x_i)=0$ holds for all $i=1,\dots,n$.

\item There exists a constant $c\in\rset$ such that
\[f(x) = c\cdot\det\begin{pmatrix}
f_0 & f_1 & \dots & f_n\\ x & x_1 & \dots & x_n
\end{pmatrix}.\]
\end{enumerate}
\end{thm}
\begin{proof}
\underline{(ii) $\Rightarrow$ (i):} Clear.

\underline{(i) $\Rightarrow$ (ii):}
If $f=0$ then $c=0$ so the assertion holds.
If $f\neq 0$ then there is a $x_0\in\cX\setminus\{x_1,\dots,x_n\}$ such that $f(x)\neq 0$.
Then also the determinant in (ii) is non-zero and we can choose $c$ such that both $f$ and the scaled determinant coincide also in $x_0$.
By \Cref{cor:uniqueDeter} a polynomial is uniquely determined by $x_0,\dots,x_n$ which shows that (ii) is one and hence the only possible polynomial which fulfills (i).
\end{proof}

So far we treated general T-systems. For further properties we go to continuous T-systems.

\begin{dfn}
Let $n\in\nset$, $\cF$ be a continuous T-system on $\cX\subseteq\rset$ an interval, $f\in\lin\cF$, and let $x_0$ be a zero of $f$. Then $x_0\in\inter\cX$ is called a \emph{non-nodal} zero if $f$ does not change sign at $x_0$. Otherwise the zero $x_0$ is called \emph{nodal}, i.e., either $f$ changes signs at $x_0$ or $x_0$ is a boundary point of $\cX$.
\end{dfn}

The following result bounds the number of nodal and non-nodal zeros.

\begin{thm}[see {\cite{krein51}} or e.g.\ {\cite[p.\ 34, Thm.\ 1.1]{kreinMarkovMomentProblem}}]\label{thm:zeros1}
Let $n\in\nset$, $\cF$ be a continuous T-system of order $n$ on $\cX = [a,b]$ with $a<b$.

If $f\in\lin\cF$ has $k\in\nset_0$ non-nodal zeros and $l\in\nset_0$ nodal zeros in $\cX$ then $2k + l\leq n$.
\end{thm}

The proof is adapted from {\cite[Thm.\ 1.1]{kreinMarkovMomentProblem}}.

\begin{proof}
\underline{$\cX=[a,b]$ and $k=0$:} If $f\in\lin\cF$ has $l$ zeros then $n \geq l$ by \Cref{dfn:tSystem}.

\underline{$\cX=[a,b]$ and $k\geq 1$:} Let $x_1,\dots,x_n\in\cX$ be the zeros of $f$. Set
\[M_i := \max_{x_{i-1} \leq x\leq x_i} |f(x)|\]
for $i=1,\dots,k+l+1$ with $t_0 = a$ and $x_{k+l+1}=b$. Additionally, set
\[m := \frac{1}{2}\min_{i=1,\dots,k+l+1} M_i >0.\]

We construct a polynomial $g_1\in\lin\cF$ such that
\begin{enumerate}[\qquad (a)]
\item $g_1$ has the values $m$ at the non-nodal zeros $x_i$ of $f$ with $f\geq 0$ in a neighborhood of $x_i$,

\item $g_1$ has the values $-m$ at the non-nodal zeros $x_i$ of $f$ with $f\leq 0$ in a neighborhood of $x_i$, and

\item $g_1$ vanishes at all nodal zeros $x_i$.
\end{enumerate}
After renumbering the $x_i$'s we can assume $x_1,\dots,x_{k_1}$ fulfill (a), $x_{k_1+1},\dots,x_{k_1+k_2}$ fulfill (b), and $x_{k_1+k_2+1},\dots,x_{k_1+k_2+l}$ fulfill (c) with $k_1+k_2=k$. By \Cref{dfn:tSystem} we have $k+l\leq n$ and hence by \Cref{lem:determinant} we have that
\[\begin{pmatrix}
m\\ \vdots\\ m\\ -m\\ \vdots\\ -m\\ 0\\ \vdots\\ 0
\end{pmatrix} = \begin{pmatrix}
f_0(x_1) & f_1(x_1) & \dots & f_n(x_1)\\
\vdots & \vdots & & \vdots\\
f_0(x_{k+l}) & f_1(x_{k+l}) & \dots & f_n(x_{k+l})
\end{pmatrix}\cdot \begin{pmatrix}
\beta_0\\ \vdots\\ \beta_n
\end{pmatrix} \]
has at least one solution, say $\beta_0 = b_0$, \dots, $\beta_n = b_n$. Then $g_1 = \sum_{i=0}^n b_i\cdot f_i\in\lin\cF $ fulfills (a) to (c).

Set
\[\rho := \frac{m}{2\cdot \|g_1\|_\infty}\]
and let $g_2 := f - g_1$.

We show that to each non-nodal zero $x_i$ of $f$ there correspond two zeros of $g_2$:
Let $x_i$ be a non-nodal zero of $f$ with $f\geq 0$ in a neighborhood of $x_i$, say.
We can find a point $y_i\in (x_{i-1},x_i)$ and a point $y_{i+1}\in (x_i,x_{i+1})$ such that
\[f(y_i) = M_i > m \qquad\text{and}\qquad f(y_{i+1}) = M_{i+1} > m.\]
Therefore, $g_2(y_i) > 0$ and $g_2(y_{i+1}) > 0$. Since $g_2(x_i) = -\rho\cdot m < 0$ if follows that $g_2$ as a zero both in $(y_i,x_i)$ and $(x_i,y_{i+1})$.

Additionally, $g_2$ also vanishes at all nodal zeros of $f$ and so has at least $2k+l$ distinct zeros. Therefore, by \Cref{dfn:tSystem} we have $2k+l\leq n$.
\end{proof}

\begin{cor}\label{cor:zeros1}
\Cref{thm:zeros1} also holds for sets $\cX\subseteq\rset$ of the form
\begin{enumerate}[(i)]
\item $\cX = (a,b)$, $[a,b)$, $(a,b]$ with $a<b$,

\item $\cX = (a,\infty)$, $[a,\infty)$, $(-\infty,b)$, $(-\infty,b]$,

\item $\cX = \{x_1,\dots,x_k\}\subseteq\rset$ with $k\geq n+1$ and $x_1 < \dots < x_k$, and

\item countable unions of (i) to (iii).
\end{enumerate}
\end{cor}
\begin{proof}
\underline{$\cX=[0,\infty)$:} Let $0\leq x_1 < \dots < x_k$ be the zeros of $f$ in $[0,\infty)$. Since every T-system on $[0,\infty)$ is also a T-system on $[0,b]$ for any $b>0$ the claim follows from \Cref{thm:zeros1} with $b=x_k + 1$.

For the other assertions adapt (if necessary) the proof of \Cref{thm:zeros1}.
\end{proof}

\begin{dfn}\label{dfn:index}
Let $x\in [a,b]$ with $a\leq b$. We define the \emph{index} $\varepsilon(x)$ by
\[\varepsilon(x) := \begin{cases}
2 &\text{if}\ x\in (a,b),\\
1 &\text{if}\ x=a\ \text{or}\ b.
\end{cases}\]
The same definition holds for sets $\cX$ as in \Cref{cor:zeros1}.
\end{dfn}

\begin{dfn}
Let $n\in\nset$ and $\cF$ be a T-system of order $n$ on some set $\cX$. We define
\begin{align*}
(\lin\cF)^e &:= \left\{ \sum_{i=0}^n a_i\cdot f_i \,\middle|\, \sum_{i=0}^n a_i^2 = 1\right\},\\
(\lin\cF)_+ &:= \left\{ f\in\cF \,\middle|\, f\geq 0\ \text{on}\ \cX\right\},
\intertext{and}
(\lin\cF)_+^e &:= (\lin\cF)^e \cap (\lin\cF)_+.
\end{align*}
\end{dfn}

With these definitions we can prove the following existence criteria for non-negative polynomials in a T-systems on $[a,b]$.

\begin{thm}\label{thm:zeros2}
Let $n\in\nset$, $\cF$ be a continuous T-system on $\cX=[a,b]$, and $x_1,\dots,x_m\in\cX$.
The following are equivalent:
\begin{enumerate}[(i)]
\item The points $x_1,\dots,x_m$ are zeros of a non-negative polynomial $f\in\lin\cF$.

\item $\displaystyle\sum_{i=1}^m \varepsilon(x_i) \leq n$.
\end{enumerate}
\end{thm}

The proof is adapted from {\cite[p.\ 35, Thm.\ 1.2]{kreinMarkovMomentProblem}}.

\begin{proof}
``(i) $\Rightarrow$ (ii)'' is \Cref{thm:zeros1} and we therefore only have to prove ``(ii) $\Rightarrow$ (i)''.

\underline{Case I:} At first assume that $a < x_1 < \dots < x_m < b$ and $\sum_{i=0}^m \varepsilon(x_i) = 2m = n$.
If $2m < n$ then add $k$ additional points $x_{m+1},\dots,x_{m+k}$ such that $2m + 2k = n$ and $x_m < x_{m+1} < \dots < x_{m+k} < b$.

Select a sequence of points $(x_1^{(j)},\dots,x_m^{(j)})\in\rset^m$, $j\in\nset$, such that
\[a < x_1 < x_1^{(j)} < \dots < x_m < x_m^{(j)} < b\]
for all $j\in\nset$ and $\lim_{j\to\infty} x_i^{(j)} = x_i$ for all $i=1,\dots,m$. Set
\begin{equation}\label{eq:gjdfns}
g_j(x) := c_j\cdot \det\begin{pmatrix}
f_0 & f_1 & f_2 & \dots & f_{2m-1} & f_{2m}\\
x & x_1 & x_1^{(j)} & \dots & x_m & x_m^{(j)}
\end{pmatrix} \quad\in (\lin\cF)^e
\end{equation}
for some $c_j>0$.
Since $(\lin\cF)^e$ is compact we can assume that $g_j$ converges to some $g_0\in (\lin\cF)^e$.
Then $g_0$ has $x_1,\dots,x_m$ as zeros with $\varepsilon(x_i)=2$ and $g_0$ is non-negative since $g_j>0$ on $[a,x_1)$, $(x_1^{(j)},x_2)$, \dots, $(x_{m-1}^{(j)}, x_m)$, and $(x_m^{(j)},b]$ as well as $g_j<0$ on $(x_1,x_1^{(j)})$, $(x_2,x_2^{(j)})$, \dots, $(x_m,x_m^{(j)})$.

\underline{Case II:}
If $a = x_1 < x_2 < \dots < x_m < b$ with $\sum_{i=1}^m \varepsilon(x_i) = 2m-1 = n$ the only modification required in case I is to replace (\ref{eq:gjdfns}) by
\[g_j(x) := -c_j\cdot \det \begin{pmatrix}
f_0 & f_1 & f_2 & f_3 & \dots & f_{2m-2} & f_{2m-1}\\
x & a & x_2 & x_2^{(j)} & \dots & x_m & x_m^{(j)}
\end{pmatrix} \quad\in (\lin\cF)^e\]
with some normalizing factor $c_j > 0$.

\underline{Case III:}
The procedure is similar if $x_m = b$ and $\sum_{i=1}^m \varepsilon(x_i) = n$.
\end{proof}

\begin{rem}\label{rem:kreinError}
\Cref{thm:zeros2} appears in {\cite[p.\ 35, Thm.\ 1.2]{kreinMarkovMomentProblem}} in a stronger version.
In {\cite[p.\ 35, Thm.\ 1.2]{kreinMarkovMomentProblem}} Krein claims that the $x_1,\dots,x_m$ are the only zeros of some non-negative $f\in\lin\cF$.
This holds when $n=2m + 2p$ for some $p>0$.
To see this add to $x_1,\dots,x_m$ in (\ref{eq:gjdfns}) points $x_{m+1},\dots,x_{m+p}\in\inter\cX\setminus\{x_1,\dots,x_m\}$ and get $g_0$.
Hence, $g_0\geq 0$ has exactly the zeros $x_1,\dots,x_{m+p}$.
Then construct in a similar way $g_0'$ with the zeros $x_1,\dots,x_m,x'_{m+1},\dots,x'_{m+p}$ with $x'_{m+1},\dots,x'_{m+p}\in\inter\cX\setminus\{x_1,\dots,x_{m+p}\}$.
Hence, $g_0 + g_0'\geq 0$ has only the zeros $x_1,\dots,x_m$.

A similar construction works for $n=2m+1$ with or without end points $a$ or $b$.

However, Krein misses that for $n=2m$ and when one end point is contained in $x_1,\dots,x_m$ then it might happen that also the other end point must appear.
In \cite[p.\ 28, Thm.\ 5.1]{karlinStuddenTSystemsBook} additional conditions are given which ensure that $x_1,\dots,x_m$ are the only zeros of some $f\geq 0$.

For example if also $\{f_i\}_{i=0}^{n-1}$ is a T-system then is can be ensured that $x_1,\dots,x_m$ are the only zeros of some non-negative polynomial $f\in\lin\cF$, see \cite[p.\ 28, Thm.\ 5.1 (b-i)]{karlinStuddenTSystemsBook}.
For our main example(s), the algebraic polynomials with gaps, this holds.

The same problem appears in \cite[p.\ 36, Thm.\ 1.3]{kreinMarkovMomentProblem}.
A weaker but correct version is given in \Cref{thm:zeros3} below.\exmsymbol
\end{rem}

\begin{rem}\label{rem:doublezeros}
Assume that in \Cref{thm:zeros2} we have additionally that $f_0,\dots,f_n\in C^1([a,b])$. Then in (\ref{eq:gjdfns}) we can set $x_i^{(j)} = x_i + j^{-1}$ for all $i=0,\dots,m$ and $j\gg 1$. For $j\to\infty$ with $c_j = j^m$ we then get
\begin{align}
g_0(x) &= \lim_{j\to\infty} j^m\cdot\det \begin{pmatrix}
f_0 & f_1 & f_2 & \dots & f_{2m-1} & f_{2m}\\
x & x_1 & x_1 + j^{-1} & \dots & x_m & x_m + j^{-1}
\end{pmatrix}\notag\\
&= \lim_{j\to\infty} j^m\cdot\det \begin{pmatrix}
f_0(x) & \dots & f_{2m}(x)\\
f_0(x_1) & \dots & f_{2m}(x_1)\\
f_0(x_1+j^{-1}) & \dots & f_{2m}(x_1+j^{-1})\\
\vdots & & \vdots\\
f_0(x_m) & \dots & f_{2m}(x_m)\\
f_0(x_m+j^{-1}) & \dots & f_{2m}(x_m+j^{-1})\\
\end{pmatrix}\notag\\
&= \lim_{j\to\infty} \det \begin{pmatrix}
f_0(x) & \dots & f_{2m}(x)\\
f_0(x_1) & \dots & f_{2m}(x_1)\\
\frac{f_0(x_1+j^{-1})-f_0(x_1)}{j^{-1}} & \dots & \frac{f_{2m}(x_1+j^{-1})-f_{2m}(x_1)}{j^{-1}}\\
\vdots & & \vdots\\
f_0(x_m) & \dots & f_{2m}(x_m)\\
\frac{f_0(x_m+j^{-1})-f_0(x_m)}{j^{-1}} & \dots & \frac{f_{2m}(x_m+j^{-1})-f_{2m}(x_m)}{j^{-1}}
\end{pmatrix}\label{eq:generalg0construction}\\
&= \begin{pmatrix}
f_0(x) & \dots & f_{2m}(x)\\
f_0(x_1) & \dots & f_{2m}(x_1)\\
f_0'(x_1) & \dots & f_{2m}'(x_1)\\
\vdots & & \vdots\\
f_0(x_m) & \dots & f_{2m}(x_m)\\
f_0'(x_m) & \dots & f_{2m}'(x_m)
\end{pmatrix},\notag
\end{align}
i.e., double zeros are included by including the values $f_i'(x_j)$.
Therefore, whenever we have $C^1$-functions in $\cF = \{f_i\}_{i=0}^n$ and $x_i = x_{i+1}$ we define
\begin{equation}\label{eq:doublezeroDfn}
\begin{pmatrix}
f_0 & \dots & f_{i-1} & f_i & f_{i+1} & f_{i+2} & \dots & f_n\\
x_0 & \dots & x_{i-1} & (x_i & x_i) & x_{i+2} & \dots & x_n
\end{pmatrix} := \begin{pmatrix}
f_0(x_0) & \dots & f_n(x_0)\\
\vdots & & \vdots\\
f_0(x_{i-1}) & \dots & f_n(x_{i-1})\\
f_0(x_i) & \dots & f_n(x_i)\\
f_0'(x_i) & \dots & f_n'(x_i)\\
f_0(x_{i+2}) & \dots & f_n(x_{i+2})\\
\vdots & & \vdots\\
f_0(x_n) & \dots & f_n(x_n)
\end{pmatrix}
\end{equation}
and equivalently when $x_j = x_{j+1}$, $x_k = x_{k+1}$, \dots\ for additional entries.
We use the additional brackets ``$($'' and ``$)$'' to indicate that $x_i$ is inserted in the $f_0,\dots,f_n$ and then also into $f_0',\dots,f_n'$ to distinguish (\ref{eq:doublezeroDfn}) from \Cref{dfn:kreinMatrix} to avoid confusion. Hence,
\[\det\begin{pmatrix}
f_0 & \dots & f_{i-1} & f_i & f_{i+1} & f_{i+2} & \dots & f_n\\
x_0 & \dots & x_{i-1} & x_i & x_i & x_{i+2} & \dots & x_n
\end{pmatrix} = 0\]
since in two rows $x_i$ is inserted into $f_0,\dots,f_n$, while in
\[\begin{pmatrix}
f_0 & \dots & f_{i-1} & f_i & f_{i+1} & f_{i+2} & \dots & f_n\\
x_0 & \dots & x_{i-1} & (x_i & x_i) & x_{i+2} & \dots & x_n
\end{pmatrix}\]
indicates that $x_i$ is inserted in $f_0,\dots,f_n$ and then also into $f_0',\dots,f_n'$.

Extending this to zeros of order $k$ for $C^{k+1}$-functions is straight forward and we leave it to the reader to write down the formulas and their proofs.
Similar to (\ref{eq:doublezeroDfn}) we write for any $a\leq x_0 \leq x_1 \leq \dots \leq x_n\leq b$ the matrix as
\[\begin{pmatrix}
f_0 & f_1 & \dots & f_n\\
x_0 & x_1 & \dots & x_n
\end{pmatrix}^*\]
when $f_0,\dots, f_n$ are sufficiently differentiable.

We often want to express polynomials $f\in\lin\cF$ as determinants (\ref{eq:gjdfns}) only by knowing their zeros $x_1,\dots,x_k$. If arbitrary multiplicities appear we only have $x_1 \leq x_2 \leq \dots \leq x_n$ where we include zeros multiple times according to their multiplicities. Hence, for
\[x_0 = \dots = x_{i_1} < x_{i_1+1} = \dots = x_{i_2} < \dots < x_{i_k+1} = \dots = x_n\]
we introduce a simpler notation to write down (\ref{eq:doublezeroDfn}):
\begin{equation}\label{eq:doublezeroDfn2}
\left(\begin{array}{c|cccc}
f_0 & f_1 & f_2 & \dots & f_n\\
x & x_1 & x_2 & \dots & x_n
\end{array}\right) :=
\begin{pmatrix}
f_0 & f_1\; \dots\; f_{i_1} & f_{i_1+1} \;\dots\; f_{i_2} & \dots & f_{i_k+1} \;\dots\; f_{i_k+1}\\
x & (x_1\; \dots\; x_{i_1}) & (x_{i_1+1} \;\dots\; x_{i_2}) & \dots & (x_{i_k+1} \;\dots\; \;\; x_n)\;\;\,
\end{pmatrix}.
\end{equation}
Clearly $(\ref{eq:doublezeroDfn2}) \in\lin\cF$. For (\ref{eq:doublezeroDfn2}) to be well-defined we need $\cF\subset C^{m-1}$ where $m$ is the largest multiplicity of any zero.
However, the procedure (\ref{eq:generalg0construction}) can lead to the zero polynomial.
We have to introduce ET-systems, see \Cref{sec:et} and \Cref{dfn:etSystem}.

In \Cref{thm:zeros2} we did not need the condition $\cF\subset C^m$ for some $m\geq 1$. The limit $g_0$ of the $g_j$ in (\ref{eq:gjdfns}) does not need the unique $f_0',\dots,f_n'$ and therefore the limit needs not to be unique.\exmsymbol
\end{rem}

\begin{cor}\label{cor:zeros2}
\Cref{thm:zeros2} also holds for intervals $\cX\subseteq\rset$, i.e.,
\begin{equation}\label{eq:intervals}
\cX = (a,b),\ (a,b],\ [a,b),\ [a,b],\ (a,\infty),\ [a,\infty),\ (-\infty,b),\ (-\infty,b],\ \text{and}\ \rset \qquad\text{with}\ a<b.
\end{equation}
\end{cor}
\begin{proof}
We have that ``(i) $\Rightarrow$ (ii)'' follows from \Cref{cor:zeros1}.

For ``(ii) $\Rightarrow$ (i)'' we apply \Cref{thm:zeros2} on $[\min_i x_i,\max_i x_i]$ and \Cref{cor:zeros1} assures that no additional zeros appear in $\cX$.
\end{proof}

We will give a sharper version of \Cref{thm:zeros1}, see also \Cref{rem:kreinError}.

\begin{thm}\label{thm:zeros3}
Let $n\in\nset$ and $\cF$ be a continuous T-system on $\cX=[a,b]$.
Additionally, let $x_1,\dots,x_k\in\cX$ and $y_1,\dots,y_l\in\cX$ be pairwise distinct points.
The following are equivalent:
\begin{enumerate}[(i)]
\item There exists a polynomial $f\in\lin\cF$ such that
\begin{enumerate}[(a)]
\item $x_1,\dots,x_k$ are the non-nodal zeros of $f$ and
\item $y_1,\dots,y_l$ are the nodal zeros of $f$.
\end{enumerate}

\item $2k+l \leq n$.
\end{enumerate}
\end{thm}
\begin{proof}
\underline{(i) $\Rightarrow$ (ii):} That is \Cref{thm:zeros1}.

\underline{(ii) $\Rightarrow$ (i):} Adapt the proof and especially the $g_j$'s in (\ref{eq:gjdfns}) of \Cref{thm:zeros2} accordingly.
Let $z_1 < \dots < z_{k+l}$ be the $x_i$'s and $y_i$'s together ordered by size.
Then in $g_j$ treat every nodal $z_i$ like the endpoint $a$ or $b$, i.e., include it only once in the determinant, and insert for every non-nodal point $z_i$ the point $z_i$ and the sequence $z_i^{(j)}\in (z_i,z_{i+1})$ with $\lim_{j\to\infty} z_i^{(j)} = z_i$.
\end{proof}

\begin{cor}
\Cref{thm:zeros3} also holds for sets $\cX\subseteq\rset$ of the form
\begin{enumerate}[(i)]
\item $\cX = (a,b)$, $[a,b)$, $(a,b]$ with $a<b$,

\item $\cX = (a,\infty)$, $[a,\infty)$, $(-\infty,b)$, $(-\infty,b]$,

\item $\cX = \{x_1,\dots,x_k\}\subseteq\rset$ with $k\geq n+1$ and $x_1 < \dots < x_k$, and

\item countable unions of (i) to (iii).
\end{enumerate}
\end{cor}
\begin{proof}
In the adapted proof and the $g_j$'s in (\ref{eq:gjdfns}) of \Cref{thm:zeros2} we do not need to have non-negativity, i.e., in the $g_j$'s sign changes at the $y_i$'s are allowed (and even required).
\end{proof}

\subsection{ET-Systems}
\label{sec:et}

\begin{dfn}\label{dfn:etSystem}
Let $n\in\nset$ and let $\cF = \{f_i\}_{i=0}^n\subset C^n([a,b])$ be a T-system of order $n$ on $[a,b]$ with $a<b$. $\cF$ is called an \emph{extended Tchebycheff system} (\emph{ET-system}) \emph{of order $n$} if any polynomial $f\in\lin\cF\setminus\{0\}$ has at most $n$ zeros counting algebraic multiplicities.
\end{dfn}

For notation of the matrices
\[\begin{pmatrix}
f_0 & f_1 & \dots & f_n\\
x_0 & x_1 & \dots & x_n
\end{pmatrix}^*\]
for $a\leq x_0 \leq x_1 \leq \dots x_n\leq b$ see the previous \Cref{rem:doublezeros}.

\begin{cor}[\cite{krein51} or e.g.\ {\cite[p.\ 37, P.1.1]{kreinMarkovMomentProblem}}]
Let $n\in\nset$ and $\cF=\{f_i\}_{i=0}^n\subset C^n([a,b])$. The following are equivalent:
\begin{enumerate}[(i)]
\item $\cF$ is an ET-system.

\item We have
\[\det\begin{pmatrix}
f_0 & f_1 & \dots & f_n\\
x_0 & x_1 & \dots & x_n
\end{pmatrix}^* \neq 0\]
for every $a\leq x_0 \leq x_1 \leq \dots \leq x_n\leq b$.
\end{enumerate}
\end{cor}
\begin{proof}
Follows immediately from \Cref{rem:doublezeros}.
\end{proof}

\begin{exm}[see e.g.\ {\cite[p.\ 19, Exm.\ 12]{karlinStuddenTSystemsBook}}]\label{exm:generalETsystem}
Let $n\in\nset$ and $g_0,\dots,g_n\in C^n([a,b])$ such that $g_0,\dots,g_n>0$ on $[a,b]$ with $a<b$.
Define
\begin{align*}
f_0(x) &:= g_0(x)\\
f_1(x) &:= g_0(x)\cdot \int_a^x g_1(y_1)~\diff y_1\\
f_2(x) &:= g_0(x)\cdot \int_a^x g_1(y_1)\cdot\int_a^{y_1} g_2(y_2)~\diff y_2~\diff y_1\\
&\;\;\;\vdots \\
f_n(x) &:= g_0(x)\cdot \int_a^x g_1(y_1)\cdot \int_a^{y_1} g_2(y_2)~ {\dots} \int_a^{y_{n-1}} g_n(y_n)~\diff y_n~ \dots ~\diff y_2~\diff y_1.
\end{align*}
Then $\{f_i\}_{i=0}^n$ is an ET-system on $[a,b]$.
\exmsymbol
\end{exm}

\begin{exm}\label{exm:nonETalg}
Let $\cF= \{1,x,x^3\}$ on $[0,b]$, $b>0$. Then $\cF$ is a T-system (Example \ref{exm:algPoly}(b)) but not an ET-system.
To see this let $x_0 = x_1 = x_2 = 0$, then
\[\begin{pmatrix}
f_0 & f_1 & f_2\\
0 & 0 & 0
\end{pmatrix}^* = \begin{pmatrix}
1 & 0 & 0\\
0 & 1 & 0\\
0 & 0 & 0
\end{pmatrix}.\]
This shows that $\cF$ is not an ET-system.\exmsymbol
\end{exm}

In the previous example the position $x=0$ prevents the T-system to be a ET-system. If $x=0$ is removed then it is an ET-system.

\begin{exm}\label{exm:algETsystem}
Let $\alpha_0,\dots,\alpha_n\in\nset_0$ with $\alpha_0 < \alpha_1 < \dots < \alpha_n$.
Then $\cF = \{x^{\alpha_i}\}_{i=0}^n$ on $[a,\infty)$ with $a>0$ is an ET-system.
For $n=2m$ and $a < x_1 < x_2 < \dots < x_m$ we often encounter a specific polynomial structure and hence we write it down explicitly once:
\begin{align}
&\det\begin{pmatrix}
x^{\alpha_0} & x^{\alpha_1} & x^{\alpha_2} & \dots & x^{\alpha_{2m-1}} & x^{\alpha_{2m}}\\
x & (x_1 & x_1) & \dots & (x_m & x_m)
\end{pmatrix}\notag\\
&= \lim_{\varepsilon\to 0} \varepsilon^{-m}\cdot \det\begin{pmatrix}
x^{\alpha_0} & x^{\alpha_1} & x^{\alpha_2} & \dots & x^{\alpha_{2m-1}} & x^{\alpha_{2m}}\\
x & x_1 & x_1+\varepsilon & \dots & x_m & x_m+\varepsilon
\end{pmatrix}\notag\\
&=\lim_{\varepsilon\to 0} \left[\prod_{i=1}^m (x_i-x)(x_i+\varepsilon-x)\right]\cdot \left[\prod_{1\leq i<j\leq m} (x_j-x_i)^2(x_j-x_i-\varepsilon)(x_j+\varepsilon-x_i) \right]\label{eq:schurPolyRepr}\\
&\qquad\times s_\alpha(x,x_1,x_1+\varepsilon,\dots,x_m,x_m+\varepsilon)\notag\\
&= \prod_{i=1}^m (x_i-x)^2 \cdot \prod_{1\leq i<j\leq m} (x_j-x_i)^4\cdot s_\alpha(x,x_1,x_1,\dots,x_m,x_m)\notag
\end{align}
where $s_\alpha$ is the Schur polynomial of $\alpha = (\alpha_0,\dots,\alpha_n)$ \cite{macdonSymFuncHallPoly}.
Hence, $s_\alpha(x,x_1,x_1,\dots,x_m,x_m)$ is not divisible by some $(x_i-x)$.

In fact, this is a special case of \Cref{exm:generalETsystem}.
With \Cref{exm:generalETsystem} we can even allow $-\infty < \alpha_0 < \alpha_1 < \dots < \alpha_n < \infty$ to be reals since $a>0$.
\exmsymbol
\end{exm}
\begin{proof}
Combine the induction
\[f^{(m+1)}(x) = \lim_{h\to 0} \frac{f^{(m)}(x+h) - f^{(m)}(x)}{h}\]
and
\[\det \begin{pmatrix}
x^{\alpha_0} & \dots & x^{\alpha_n}\\
x_0 & \dots & x_n
\end{pmatrix} =
\prod_{0\leq i<j\leq n} (x_j - x_i)\cdot s_\alpha(x_0,\dots,x_n)\]
where $s_\alpha$ is the Schur polynomial of $\alpha = (\alpha_0,\dots,\alpha_n)$.
\end{proof}

\begin{exm}
Let $n\in\nset$. Then the T-system $\cF = \{x^i\}_{i=0}^n$ on $\rset$ is an ET-system.\exmsymbol
\end{exm}

\section{Sparse Positivstellensätze and Nichtnegativstellensätze}
\label{sec:sparsePosNiNeg}

In this section we present the Positivestellensatz for T-systems of Karlin (\Cref{thm:karlinab} and \Cref{cor:posabDecomp}).
We show their application to gain algebraic sparse Positivestellensätze and Nichtnegativstellensätze.
They will be used to solve sparse moment problems.

\subsection{Sparse Positivstellensätze and Nichtnegativstellensätze on $[a,b]$ for general T-Systems}
\label{sec:posabGeneral}

For what follows we want to remind the reader of the index $\varepsilon(x)$ of a point $x$, see \Cref{dfn:index}.

\begin{dfn}
Let $\cZ\subset [a,b]$. We say \emph{$\cZ$ has index $n$} if $\sum_{x\in\cZ} \varepsilon(x) = n$.
The same definition holds for sets $\cX$ as in \Cref{cor:zeros1}.
\end{dfn}

Because of its importance and since it was first proved in full generality by Karlin in \cite{karlin63} we call the following result Karlin's Theorem.

\begin{karthm}[{\cite{karlin63}} or e.g.\ {\cite[p.\ 66, Thm.\ 10.1]{karlinStuddenTSystemsBook}}]\label{thm:karlinab}
Let $n\in\nset$, $\cF = \{f_i\}_{i=0}^n$ be a continuous T-system of order $n$ on $[a,b]$ with $a<b$, and let $f\in C([a,b])$ with $f>0$ on $[a,b]$.
The following hold:
\begin{enumerate}[(i)]
\item There exists a unique polynomial $f_*\in\lin\cF$ such that
\begin{enumerate}[(a)]
\item $f(x) \geq f_*(x) \geq 0$ for all $x\in [a,b]$,

\item $f_*$ vanishes on a set with index $n$,

\item the function $f-f_*$ vanishes at least once between each pair of adjacent zeros of $f_*$,

\item the function $f-f_*$ vanishes at least once between the largest zero of $f_*$ and the end point $b$,

\item $f_*(b)>0$.
\end{enumerate}

\item There exists a unique polynomial $f^*\in\lin\cF$ which satisfies the conditions (a)--(d) of (i), and
\begin{enumerate}[(a')]\setcounter{enumii}{4}
\item $f^*(b) = 0$.
\end{enumerate}
\end{enumerate}
\end{karthm}
\begin{proof}
See e.g.\ \cite[p.\ 68--71]{karlinStuddenTSystemsBook}.
\end{proof}

Note, in the previous result we do not need to have $f\in\lin\cF$. The function $f$ only needs to be continuous and strictly positive on $[a,b]$.

An earlier version of \Cref{thm:karlinab} is a lemma by Markov \cite{markov84}, see also \cite[p.\ 80]{shohat43}.

For the same reason as for \Cref{thm:karlinab} we call the following immediate consequence Karlin's Corollary. It is the T-system Positivstellensatz by Karlin and will be used to generate (algebraic) Positivestellensätze.

\begin{karcor}[{\cite{karlin63}} or e.g.\ {\cite[p.\ 71, Cor.\ 10.1(a)]{karlinStuddenTSystemsBook}}]\label{cor:posabDecomp}
Let $n\in\nset$, $\cF$ be a continuous T-system of order $n$ on $[a,b]$ with $a<b$, and let $f\in\lin\cF$ with $f>0$ on $[a,b]$.
Then there exists a unique representation
\[ f = f_* + f^*\]
with $f_*,f^*\in\lin\cF$ such that
\begin{enumerate}[(i)]
\item $f_*,f^*\geq 0$ on $[a,b]$,

\item the zeros of $f_*$ and $f^*$ each are sets of index $n$,

\item the zeros of $f_*$ and $f^*$ strictly interlace, and

\item $f^*(b) = 0$.
\end{enumerate}
\end{karcor}
\begin{proof}
Let $f_*$ be the unique $f_*$ from \Cref{thm:karlinab}(i). Then $f - f_*\in\lin\cF$ is a polynomial and fulfills (a)--(d), and (e') of $f^*$ in \Cref{thm:karlinab}.
But since $f^*$ is unique we have $f - f_* = f^*$.
\end{proof}

\begin{cor}[{\cite{karlin63}} or e.g.\ {\cite[p.\ 71, Cor.\ 10.1(b)]{karlinStuddenTSystemsBook}}]
Let $n\in\nset$, $\{f_i\}_{i=0}^n$ and $\{f_i\}_{i=0}^{n+1}$ be continuous T-systems of order $n$ and $n+1$ on $[a,b]$ with $a<b$. Then $f_{n+1}-(f_{n+1})_*$ and $f_{n+1} - (f_{n+1})^*$ both vanish on sets of index $n+1$.
\end{cor}

\begin{thm}[{\cite{karlin63}} or e.g.\ {\cite[p.\ 72, Thm.\ 10.2]{karlinStuddenTSystemsBook}}]
Let $n\in\nset$, $\cF = \{f_i\}_{i=0}^n$ be a continuous T-system of order $n$ on $[a,b]$ with $a<b$, and let $g_1,g_2$ be two continuous functions on $[a,b]$ such that there exists a $g'\in\lin\cF$ with
\[ g_1(x) < g'(x) < g_2(x)\]
for all $x\in [a,b]$. The following hold:
\begin{enumerate}[(i)]
\item There exists a unique polynomial $f_*\in\lin\cF$ such that
\begin{enumerate}[(a)]
\item $g_1(x) \leq f_*(x) \leq g_2(x)$ for all $x\in [a,b]$, and

\item there exist $n+1$ points $x_1<\dots<x_{n+1}$ in $[a,b]$ such that
\[ f_*(x_{n+1-i}) = \begin{cases}
g_1(x_{n+1-i}) & \text{for}\ i=1,3,5,\dots,\\
g_2(x_{n+1-i}) & \text{for}\ i=0,2,4,\dots .
\end{cases}\]
\end{enumerate}

\item There exists a unique polynomial $f^*\in\lin\cF$ such that
\begin{enumerate}[(a')]
\item $g_1(x) \leq f^*(x) \leq g_2(x)$ for all $x\in [a,b]$, and

\item there exist $n+1$ points $y_1 < \dots < y_{n+1}$ in $[a,b]$ such that
\[ f^*(y_{n+1-i}) = \begin{cases}
g_2(y_{n+1-i}) & \text{for}\ i=1,3,5,\dots,\\
g_1(y_{n+1-i}) & \text{for}\ i=0,2,4,\dots .
\end{cases}\]
\end{enumerate}
\end{enumerate}
\end{thm}
\begin{proof}
See \cite[p.\ 73]{karlinStuddenTSystemsBook}.
\end{proof}

In \Cref{thm:karlinab} and \Cref{cor:posabDecomp} we dealt with $f\in\lin\cF$ with $f>0$, i.e., they are the Positivstellensatz.
The following result allows for $f\geq 0$ and is therefore together with \Cref{cor:nonnegabDecomp} the T-system Nichtnegativstellensatz of Karlin.
We get from \Cref{cor:nonnegabDecomp} sparse algebraic Nichtnegativstellensätze (\Cref{thm:algNichtNegab}).

\begin{thm}[{\cite{karlin63}} or e.g.\ {\cite[p.\ 74, Thm.\ 10.3]{karlinStuddenTSystemsBook}}]\label{thm:karlinNichtNeg}
Let $n\in\nset$, $\cF = \{f_i\}_{i=0}^n$ be a continuous ET-system of order $n$ on $[a,b]$ with $a<b$, and let $f\in C^n([a,b])$ be such that $f\geq 0$ on $[a,b]$ and $f$ has $r < n$ zeros (counting multiplicities).
The following hold:
\begin{enumerate}[(i)]
\item There exists a unique polynomial $f_*\in\lin\cF$ such that
\begin{enumerate}[(a)]
\item $f(x)\geq f_*(x)\geq 0$ for all $x\in [a,b]$,

\item $f_*$ has $n$ zeros counting multiplicities,

\item if $x_{1} < \dots < x_{{n-r}}$ in $(a,b)$ are the zeros of $f_*$ which remain after removing the $r$ zeros of $f$ then $f - f_*$ vanishes at least twice more (counting multiplicities) in each open interval $(x_i,x_{i+1})$, $i=1,\dots,n-r-1$, and at least once more in each of the intervals $[a,x_1)$ and $(x_{n-r},b]$,

\item the zeros $x_1,\dots,x_{n-r}$ of (c) are a set of index $n-r$, and

\item $x_{n-r} < b$.
\end{enumerate}

\item There exists a unique polynomial $f^*\in\lin\cF$ satisfying the conditions (a) to (d) and (e') $x_{n-r} = b$.
\end{enumerate}
\end{thm}
\begin{proof}
See \cite[p.\ 74--75]{karlinStuddenTSystemsBook}.
\end{proof}

\begin{cor}[{\cite{karlin63}} or e.g.\ {\cite[p.\ 76, Cor.\ 10.3]{karlinStuddenTSystemsBook}}]\label{cor:nonnegabDecomp}
Let $n\in\nset$, $\cF = \{f_i\}_{i=0}^n$ be an ET-system of order $n$ on $[a,b]$ with $a<b$, and let $f\in\lin\cF$ be such that $f\geq 0$ on $[a,b]$ and $f$ has $r < n$ zeros (counting multiplicities).
Then there exists a unique representation
\[f = f_* + f^*\]
with $f_*,f^*\in\lin\cF$ such that
\begin{enumerate}[(i)]
\item $f_*, f^*\geq 0$ on $[a,b]$,

\item $f_*$ and $f^*$ have $n$ zeros (counting multiplicity) which strictly interlace if the zeros of $f$ are removed,

\item $f^*(b) =0$.
\end{enumerate}
\end{cor}

\subsection{Sparse Positivstellensätze and Nichtnegativstellensätze on $[a,b]$ for Algebraic Polynomials}

\begin{thm}[Sparse Algebraic Positivstellensatz on {$[a,b]$ with $0<a<b$}]\label{thm:algPosSatzab}
Let $n\in\nset$, $\alpha_0,\dots,\alpha_n\in\rset$ be real numbers with $\alpha_0 < \alpha_1 < \dots < \alpha_n$, and let $\cF=\{x^{\alpha_i}\}_{i=0}^n$.
Then for any $f=\sum_{i=0}^n a_i f_i\in\lin\cF$ with $f>0$ on $[a,b]$ and $a_n>0$ there exists a unique decomposition
\[f = f_* + f^*\]
with $f_*,f^*\in\lin\cF$ such that
\begin{enumerate}[(i)]
\item for $n = 2m$ there exist points $x_1,\dots,x_m,y_1,\dots,y_{m-1}\in [a,b]$ with
\[a < x_1 < y_1 < \dots < x_m < b\]
and constants $c_*,c^*>0$ with
\[f_*(x) = c_*\cdot\det \begin{pmatrix}
x^{\alpha_0} & x^{\alpha_1} & x^{\alpha_2} & \dots & x^{\alpha_{2m-1}} & x^{\alpha_{2m}}\\
x & (x_1 & x_1) & \dots & (x_m & x_m)
\end{pmatrix}\geq 0\]
and
\[f^*(x) = -c^*\cdot\det \begin{pmatrix}
x^{\alpha_0} & x^{\alpha_1} & x^{\alpha_2} & x^{\alpha_3} & \dots & x^{\alpha_{2m-2}} & x^{\alpha_{2m-1}} & x^{\alpha_{2m}}\\
x & a & (y_1 & y_1) & \dots & (y_{m-1} & y_{m-1}) & b
\end{pmatrix}\geq 0\]
for all $x\in [a,b]$, or

\item for $n = 2m+1$ there exist points $x_1,\dots,x_m,y_1,\dots,y_m\in [a,b]$ with
\[a < y_1 < x_1 < \dots < y_m < x_m < b\]
and $c_*,c^*>0$ with
\[f_*(x) = -c_*\cdot\det\begin{pmatrix}
x^{\alpha_0} & x^{\alpha_1} & x^{\alpha_2} & x^{\alpha_3} & \dots & x^{\alpha_{2m}} & x^{\alpha_{2m+1}}\\
x & a & (x_1 & x_1) & \dots & (x_m & x_m)
\end{pmatrix}\geq 0\]
and
\[f^*(x) = c^*\cdot\det\begin{pmatrix}
x^{\alpha_0} & x^{\alpha_1} & x^{\alpha_2} & \dots & x^{\alpha_{2m-1}} & x^{\alpha_{2m}} & x^{\alpha_{2m+1}}\\
x & (y_1 & y_1) & \dots & (y_m & y_m) & b
\end{pmatrix}\geq 0\]
for all $x\in [a,b]$.
\end{enumerate}
\end{thm}
\begin{proof}
By \Cref{exm:algETsystem} we have that $\cF$ on $[a,b]$ is an ET-system. Hence, \Cref{cor:posabDecomp} applies. We check both cases $n = 2m$ and $n=2m+1$ separately.

\underline{$n=2m$:}
By \Cref{cor:posabDecomp} we have that the zero set $\cZ(f^*)$ of $f^*$ has index $2m$ and contains $b$ with index $1$, i.e., $a\in\cZ(f^*)$ and all other zeros have index $2$.
Hence, $\cZ(f^*) = \{a=y_0 < y_1 < \dots < y_{m-1} < y_m = b\}$.
By \Cref{cor:posabDecomp} we have that $\cZ(f_*)$ also has index $2m$ and the zeros of $f_*$ and $f^*$ interlace.
Then the determinantal representations of $f_*$ and $f^*$ follow from \Cref{rem:doublezeros}.

\underline{$n=2m+1$:}
By \Cref{cor:posabDecomp} we have that $b\in\cZ(f^*)$ and since the index of $\cZ(f^*)$ is $2m+1$ we have that there are only double zeros $y_1,\dots,y_m\in (a,b)$ in $\cZ(f^*)$.
Similar we find that $a\in\cZ(f_*)$ since its index is odd and only double zeros $x_1,\dots,x_m\in (a,b)$ in $\cZ(f_*)$ remain.
By \Cref{cor:posabDecomp}(iii) the zeros $x_i$ and $y_i$ strictly interlace and the determinantal representation of $f_*$ and $f^*$ follow again from \Cref{rem:doublezeros}.
\end{proof}

Note, if $\alpha_0,\dots,\alpha_n\in\nset_0$ then by \Cref{exm:algETsystem} equation (\ref{eq:schurPolyRepr}) the algebraic polynomials $f_*$ and $f^*$ can also be written down with Schur polynomials.

\Cref{thm:algPosSatzab} does not to hold for $a=0$ and $\alpha_0>0$ or $\alpha_0,\dots,\alpha_k<0$.
In case $\alpha_0>0$ the determinantal representations of $f^*$ for $n=2m$ and $f_*$ for $n=2m+1$ are the zero polynomial.
In fact, in this case $\cF$ is not even a T-system since in \Cref{lem:determinant} the determinant contains a zero column if $x_0 = 0$.
We need to have $\alpha_0=0$ ($x^{\alpha_0} = 1$) to let $a=0$.
For $\alpha_0,\dots,\alpha_k<0$ we have singularities at $x=0$ and hence no T-system.

\begin{cor}\label{cor:azero}
If $\alpha_0=0$ in \Cref{thm:algPosSatzab} then \Cref{thm:algPosSatzab} also holds with $a=0$.
\end{cor}
\begin{proof}
The determinantal representations of $f_*$ for $n=2m+1$ and $f^*$ for $n=2m$ in \Cref{thm:algPosSatzab} continuously depend on $a$.
It is sufficient to show that these representations are non-trivial (not the zero polynomial) for $a=0$.
We show this for $f_*$ in case (ii) $n=2m+1$.
The other cases are equivalent.

For $\varepsilon>0$ small enough we set
\begin{align*}
g_\varepsilon(x) &= -\varepsilon^{-m}\cdot\det\begin{pmatrix}
1 & x^{\alpha_1} & x^{\alpha_2} & x^{\alpha_3} & \dots & x^{\alpha_{2m}} & x^{\alpha_{2m+1}}\\
x & 0 & x_1 & x_1+\varepsilon & \dots & x_m & x_m+\varepsilon
\end{pmatrix}\\
&= -\varepsilon^{-m}\cdot \det\begin{pmatrix}
1 & x^{\alpha_1} & x^{\alpha_2} & \dots & x^{\alpha_{2m+1}}\\
1 & 0 & 0 & \dots & 0\\
1 & x_1^{\alpha_1} & x_1^{\alpha_2} & \dots & x_1^{\alpha_{2m+1}}\\
\vdots & \vdots & \vdots & & \vdots\\
1 & (x_{m}+\varepsilon)^{\alpha_1} & (x_{m}+\varepsilon)^{\alpha_2} & \dots & (x_{m}+\varepsilon)^{\alpha_{2m+1}}
\end{pmatrix}
\intertext{develop with respect to the second row}
&= \varepsilon^{-m}\cdot\det\begin{pmatrix}
x^{\alpha_1} & x^{\alpha_2} & \dots & x^{\alpha_{2m-1}}\\
x_1^{\alpha_1} & x_1^{\alpha_2} & \dots & x_1^{\alpha_{2m-1}}\\
\vdots & \vdots & & \vdots\\
(x_{m}+\varepsilon)^{\alpha_1} & (x_{m}+\varepsilon)^{\alpha_2} & \dots & (x_{m}+\varepsilon)^{\alpha_{2m+1}}
\end{pmatrix}\\
&= \varepsilon^{-m}\cdot\det\begin{pmatrix}
x^{\alpha_1} & x^{\alpha_2} & x^{\alpha_3} & \dots & x^{\alpha_{2m}} & x^{\alpha_{2m+1}}\\
x & x_1 & x_1+\varepsilon & \dots & x_m & x_m+\varepsilon
\end{pmatrix}.
\end{align*}
Then $x_1,x_1+\varepsilon,\dots,x_m,x_m+\varepsilon\in (0,b]$, i.e., $\{x^{\alpha_i}\}_{i=1}^n$ is an ET-system on $[a',b]$ with $0=a < a' < x_1$, see \Cref{exm:algETsystem}.
By \Cref{rem:doublezeros} the representation is not the zero polynomial which ends the proof.
\end{proof}

The \Cref{thm:algPosSatzab} is a complete description of $\inter(\lin\cF)_+$.
Since $\cF$ is continuous on the compact interval $[a,b]$ and $x^{\alpha_0}>0$ on $[a,b]$, we have that the truncated moment cone is closed and hence $(\lin\cF)_+$ and the moment cone are dual to each other.
With \Cref{thm:algPosSatzab} we can now write down the conditions for the sparse truncated Hausdorff moment problem on $[a,b]$ with $a>0$.
We are not aware of a reference for the following result.

\begin{thm}[Sparse Truncated Hausdorff Moment Problem on {$[a,b]$ with $a>0$}]\label{thm:sparseTruncHausd}
Let $n\in\nset$, $\alpha_0,\dots,\alpha_n\in[0,\infty)$ with $\alpha_0 < \dots < \alpha_n$, and $a,b$ with $0 < a < b$. Set $\cF = \{x^{\alpha_i}\}_{i=0}^n$. Then the following are equivalent:
\begin{enumerate}[(i)]
\item $L:\lin\cF\to\rset$ is a truncated $[a,b]$-moment functional.

\item $L(p)\geq 0$ holds for all
\begin{align*}
p(x) &:= \begin{cases}
\det \begin{pmatrix}
x^{\alpha_0} & x^{\alpha_1} & x^{\alpha_2} & \dots & x^{\alpha_{2m-1}} & x^{\alpha_{2m}}\\
x & (x_1 & x_1) & \dots & (x_m & x_m)
\end{pmatrix}\\
-\det\begin{pmatrix}
x^{\alpha_0} & x^{\alpha_1} & x^{\alpha_2} & x^{\alpha_3} & \dots & x^{\alpha_{2m-2}} & x^{\alpha_{2m-1}} & x^{\alpha_{2m}}\\
x & a & (x_1 & x_1) & \dots & (x_{m-1} & x_{m-1}) & b
\end{pmatrix}
\end{cases} \tag*{if $n = 2m$}
\intertext{or}
p(x) &:=\begin{cases}
-\det\begin{pmatrix}
 x^{\alpha_0} & x^{\alpha_1} & x^{\alpha_2} & x^{\alpha_3} & \dots & x^{\alpha_{2m}} & x^{\alpha_{2m+1}}\\
x & a & (x_1 & x_1) & \dots & (x_m & x_m)
\end{pmatrix}\\
\det\begin{pmatrix}
x^{\alpha_0} & x^{\alpha_1} & x^{\alpha_2} & \dots & x^{\alpha_{2m-1}} & x^{\alpha_{2m}} & x^{\alpha_{2m+1}}\\
x & (x_1 & x_1) & \dots & (x_m & x_m) & b
\end{pmatrix}
\end{cases} \tag*{if $n=2m+1$}
\end{align*}
and all $x_1,\dots,x_m$ with $a < x_1 < \dots < x_m < b$.
\end{enumerate}
\end{thm}
\begin{proof}
The implication (i) $\Rightarrow$ (ii) is clear since all given polynomials $p$ are non-negative on $[a,b]$.
It is therefore sufficient to prove (ii) $\Rightarrow$ (i).

Since $a>0$ we have that $x^{\alpha_0} > 0$ on $[a,b]$ and since $[a,b]$ is compact we have that the moment cone $((\lin\cF)_+)^*$ as the dual of the cone of non-negative (sparse) polynomials $(\lin\cF)_+$ is a closed pointed cone.

To establish $L\in ((\lin\cF)_+)^*$ it is sufficient to have $L(f)\geq 0$ for all $f\in(\lin\cF)_+$.
Let $f\in(\lin\cF)_+$.
Then for all $\varepsilon>0$ we have $f_\varepsilon := f+\varepsilon\cdot x^{\alpha_0} > 0$ on $[a,b]$, i.e., by \Cref{thm:algPosSatzab} $f_\varepsilon$ is a conic combination of the polynomials $p$ in (ii) and hence $L(f) + \varepsilon\cdot L(x^{\alpha_0}) = L(f_\varepsilon) \geq 0$ for all $\varepsilon>0$.
Since $x^{\alpha_0}>0$ on $[a,b]$ we also have that $x^{\alpha_0}$ is a conic combination of the polynomials $p$ in (ii) and therefore $L(x^{\alpha_0}) \geq 0$.
Then $L(f)\geq 0$ follows from $\varepsilon\to 0$ which proves (i).
\end{proof}

\begin{cor}
If $\alpha_0 = 0$ in \Cref{thm:sparseTruncHausd} then \Cref{thm:sparseTruncHausd} also holds with $a=0$, i.e., the following are equivalent:
\begin{enumerate}[(i)]
\item $L:\lin\cF\to\rset$ is a truncated $[0,b]$-moment functional.

\item $L(p)\geq 0$ holds for all
\begin{align*}
p(x) &:= \begin{cases}
\det \begin{pmatrix}
1 & x^{\alpha_1} & x^{\alpha_2} & \dots & x^{\alpha_{2m-1}} & x^{\alpha_{2m}}\\
x & (x_1 & x_1) & \dots & (x_m & x_m)
\end{pmatrix}\\
\det\begin{pmatrix}
x^{\alpha_1} & x^{\alpha_2} & x^{\alpha_3} & \dots & x^{\alpha_{2m-2}} & x^{\alpha_{2m-1}} & x^{\alpha_{2m}}\\
x & (x_1 & x_1) & \dots & (x_{m-1} & x_{m-1}) & b
\end{pmatrix}
\end{cases} \tag*{if $n = 2m$}
\intertext{or}
p(x) &:=\begin{cases}
\det\begin{pmatrix}
x^{\alpha_1} & x^{\alpha_2} & x^{\alpha_3} & \dots & x^{\alpha_{2m}} & x^{\alpha_{2m+1}}\\
x & (x_1 & x_1) & \dots & (x_m & x_m)
\end{pmatrix}\\
\det\begin{pmatrix}
1 & x^{\alpha_1} & x^{\alpha_2} & \dots & x^{\alpha_{2m-1}} & x^{\alpha_{2m}} & x^{\alpha_{2m+1}}\\
x & (x_1 & x_1) & \dots & (x_m & x_m) & b
\end{pmatrix}
\end{cases} \tag*{if $n=2m+1$}
\end{align*}
and all $x_1,\dots,x_m$ with $a < x_1 < \dots < x_m < b$.
\end{enumerate}
\end{cor}
\begin{proof}
Follows from \Cref{cor:azero}.
\end{proof}

For the following we want to remind the reader of the M\"untz--Sz\'asz Theorem \cite{muntz14,szasz16}.
It states that for real exponents $\alpha_0=0 < \alpha_1 < \alpha_2 < \dots$ the vector space $\lin\{x^{\alpha_i}\}_{i\in\nset_0}$ of finite linear combinations is dense in $C([0,1])$ if and only if $\sum_{i\in\nset} \frac{1}{\alpha_i} = \infty$.

We state the following only for the classical case of the interval $[0,1]$.
Other cases $[a,b]\subset [0,\infty)$ are equivalent.
We are not aware of a reference for the following result.
Hausdorff required $\alpha_i\to\infty$.
The M\"untz--Sz\'asz Theorem does not require $\alpha_i\to\infty$.
The conditions $\alpha_0=0$ and $\sum_{i\in\nset} \frac{1}{\alpha_i} = \infty$ already appear in \cite[eq.\ (17)]{hausdo21a}.

\begin{thm}[Sparse Hausdorff Moment Problem]\label{thm:sparseHausd}
Let $\{\alpha_i\}_{i\in\nset_0}\subset [0,\infty)$ with $0=\alpha_0 < \alpha_1 < \dots$ and $\sum_{i\in\nset} \frac{1}{\alpha_i} = \infty$. Let $\cF = \{x^{\alpha_i}\}_{i\in\nset_0}$.
The following are equivalent:
\begin{enumerate}[(i)]
\item $L:\lin\cF\to\rset$ is a $[0,1]$-moment functional.

\item $L(p)\geq 0$ holds for all $p\in (\lin\cF)_+$.

\item $L(p)\geq 0$ holds for all $p\in\lin\cF$ with $p>0$.

\item $L(p)\geq 0$ holds for all
\[p(x) = \begin{cases}
\det \begin{pmatrix}
1 & x^{\alpha_1} & x^{\alpha_2} & \dots & x^{\alpha_{2m-1}} & x^{\alpha_{2m}}\\
x & (x_1 & x_1) & \dots & (x_m & x_m)
\end{pmatrix},\\
\det\begin{pmatrix}
x^{\alpha_1} & x^{\alpha_2} & x^{\alpha_3} & \dots & x^{\alpha_{2m-2}} & x^{\alpha_{2m-1}} & x^{\alpha_{2m}}\\
x & (x_1 & x_1) & \dots & (x_{m-1} & x_{m-1}) & 1
\end{pmatrix},\\
\det\begin{pmatrix}
x^{\alpha_1} & x^{\alpha_2} & x^{\alpha_3} & \dots & x^{\alpha_{2m}} & x^{\alpha_{2m+1}}\\
x & (x_1 & x_1) & \dots & (x_m & x_m)
\end{pmatrix},\ \text{and}\\
\det\begin{pmatrix}
1 & x^{\alpha_1} & x^{\alpha_2} & \dots & x^{\alpha_{2m-1}} & x^{\alpha_{2m}} & x^{\alpha_{2m+1}}\\
x & (x_1 & x_1) & \dots & (x_m & x_m) & 1
\end{pmatrix}
\end{cases}\]
for all $m\in\nset$ and all $0 < x_1 < x_2 < \dots < x_m < 1$.
\end{enumerate}
\end{thm}
\begin{proof}
The implications ``(i) $\Rightarrow$ (ii) $\Leftrightarrow$ (iii)'' are clear and ``(iii) $\Leftrightarrow$ (iv)'' follows from \Cref{thm:algPosSatzab}. It is therefore sufficient to show ``(ii) $\Rightarrow$ (i)''.

Let $f\in C([0,1])$ with $f>0$.
Since $\lin\cF$ is dense in $C([0,1])$ by the M\"untz--Sz\'asz Theorem there are sequences $\{g_i\}_{i\in\nset_0}$ and $\{h_i\}_{i\in\nset_0}$ with $0 < g_i < f < h_i$ and $\|g_i-h_i\|_\infty\to 0$ as $i\to\infty$.
Hence, $L(f)\geq 0$.
Since $f\in C([0,1])$ with $f>0$ was arbitrary we have that $L(f)\geq 0$ for all $f\in C([0,1])$ with $f\geq 0$.
Then by the Riesz--Markov--Kakutani Representation Theorem we have that $L$ has a unique representing measure.
\end{proof}

The previous proof can be simplified by using Choquet's theory of adapted spaces, see \cite{choquet69} or for a more modern formulation \cite{phelpsLectChoquetTheorem} or \cite[Ch.\ 1]{schmudMomentBook}.
With that we can even remove the use of the M\"untz--Sz\'asz Theorem and therefore the condition $\sum_{i\in\nset} \frac{1}{\alpha_i} = \infty$.
Additionally, we can allow for negative exponents.
We will use this approach below and also in all other proofs from here on.
The following theorem has to our knowledge not been presented before.

\begin{thm}[General Sparse Hausdorff Moment Problem on {$[a,b]$ with $0\leq a < b$}]\label{thm:generalSparseHausd}
Let $I\subset\nset_0$ be an index set (finite or infinite), let $\{\alpha_i\}_{i\in I}$ be such that $\alpha_i\neq \alpha_j$ for all $i\neq j$ and
\begin{enumerate}[(a)]
\item if $a=0$ then $\{\alpha_i\}_{i\in I}\subset [0,\infty)$ with $\alpha_i=0$ for an $i\in I$, or

\item if $a>0$ then $\{\alpha_i\}_{i\in I}\subset\rset$.
\end{enumerate}
Let $\cF = \{x^{\alpha_i}\}_{i\in I}$.
Then the following are equivalent:
\begin{enumerate}[(i)]
\item $L:\lin\cF\to\rset$ is a Hausdorff moment functional.

\item $L(p)\geq 0$ holds for all $p\in (\lin\cF)_+$.

\item $L(p)\geq 0$ holds for all $p\in\lin\cF$ with $p>0$.

\item $L(p)\geq 0$ holds for all
\[p(x) = \begin{cases}
\det \begin{pmatrix}
x^{\alpha_{i_0}} & x^{\alpha_{i_1}} & x^{\alpha_{i_2}} & \dots & x^{\alpha_{i_{2m-1}}} & x^{\alpha_{2m}}\\
x & (x_1 & x_1) & \dots & (x_m & x_m)
\end{pmatrix}, & \text{if $|I| = 2m$ or $\infty$,}\\
\det\begin{pmatrix}
x^{\alpha_{i_1}} & x^{\alpha_{i_2}} & x^{\alpha_{i_3}} & \dots & x^{\alpha_{i_{2m-2}}} & x^{\alpha_{i_{2m-1}}} & x^{\alpha_{i_{2m}}}\\
x & (x_1 & x_1) & \dots & (x_{m-1} & x_{m-1}) & b
\end{pmatrix}, & \text{if $|I| = 2m$ or $\infty$,}\\
\det\begin{pmatrix}
x^{\alpha_{i_1}} & x^{\alpha_{i_2}} & x^{\alpha_{i_3}} & \dots & x^{\alpha_{i_{2m}}} & x^{\alpha_{i_{2m+1}}}\\
x & (x_1 & x_1) & \dots & (x_m & x_m)
\end{pmatrix},\  & \text{if $|I| = 2m+1$ or $\infty$, and}\\
\det\begin{pmatrix}
x^{\alpha_{i_0}} & x^{\alpha_{i_1}} & x^{\alpha_{i_2}} & \dots & x^{\alpha_{i_{2m-1}}} & x^{\alpha_{i_{2m}}} & x^{\alpha_{i_{2m+1}}}\\
x & (x_1 & x_1) & \dots & (x_m & x_m) & b
\end{pmatrix},  & \text{if $|I| = 2m+1$ or $\infty$,}
\end{cases}\]
for all $m\in\nset$ if $|I|=\infty$, all $0 < x_1 < x_2 < \dots < x_m < b$, and all $\alpha_{i_0} < \alpha_{i_1} < \dots < \alpha_{i_m}$ with $\alpha_{i_0}=0$ if $a=0$.
\end{enumerate}
If additionally $\sum_{i:\alpha_i\neq 0} \frac{1}{|\alpha_i|} = \infty$ then $L$ is determinate.
\end{thm}
\begin{proof}
The case $|I|<\infty$ is \Cref{thm:sparseTruncHausd}.
We therefore prove the case $|I|=\infty$.
The choice $\alpha_{i_0} < \alpha_{i_1} < \dots < \alpha_{i_m}$ with $\alpha_{i_0}=0$ if $a=0$ makes $\{x^{\alpha_{i_j}}\}_{j=0}^m$ a T-system.
The implications ``(i) $\Rightarrow$ (ii) $\Leftrightarrow$ (iii)'' are clear and ``(iii) $\Leftrightarrow$ (iv)'' is \Cref{thm:algPosSatzab}.
It is therefore sufficient to show ``(ii) $\Rightarrow$ (i)''.

The space $\lin\cF$ is an adapted space and the assertion follows therefore from \cite[Thm.\ 1.8]{schmudMomentBook}.

For the determinacy of $L$ split $\{\alpha_i\}_{i\in I}$ into positive and negative exponents. If $\sum_{i:\alpha_i\neq 0} \frac{1}{|\alpha_i|} = \infty$ then the corresponding sum over at least one group is infinite. If the sum over the positive exponents is infinite apply the M\"untz--Sz\'asz Theorem. If the sum over the negative exponents is infinite apply the M\"untz--Sz\'asz Theorem to $\{(x^{-1})^{-\alpha_i}\}_{i\in I: \alpha_i < 0}$ since $a>0$.
\end{proof}

Note, since $[a,b]$ is compact the fact that $\{x^{\alpha_i}\}_{i\in I}$ is an adapted space is trivial.

In the previous results we only needed the description of all strictly positive polynomials.
The non-negative polynomials are described in the following result.
Again, we are not aware of a reference.

\begin{thm}[Sparse Algebraic Nichtnegativstellensatz on {$[a,b]$ with $0<a<b$}]\label{thm:algNichtNegab}
Let $n\in\nset$, $\alpha_0,\dots,\alpha_n\in\rset$ be real numbers with $\alpha_0 < \alpha_1 < \dots < \alpha_n$, and let $\cF = \{x^{\alpha_i}\}_{i=0}^n$. Let $f\in\lin\cF$ with $f\geq 0$ on $[a,b]$. Then there exist points $x_1,\dots,x_n,y_1,\dots,y_n\in [a,b]$ (not necessarily distinct) with $y_n=b$ which include the zeros of $f$ with multiplicities and there exist constants $c_*,c^*\in\rset$ such that
\[f = f_* + f^*\]
with $f_*,f^*\in\lin\cF$, $f_*,f^*\geq 0$ on $[a,b]$, and the polynomials $f_*$ and $f^*$ are given by
\[f_*(x) = c_*\cdot\det\left(\begin{array}{c|ccc}
f_0 & f_1 & \dots & f_n\\
x & x_1 & \dots & x_n
\end{array}\right)
\qquad\text{and}\qquad
f^*(x) = c_*\cdot\det\left(\begin{array}{c|ccc}
f_0 & f_1 & \dots & f_n\\
x & y_1 & \dots & y_n
\end{array}\right)\]
for all $x\in [a,b]$.

Removing the zeros of $f$ from $x_1,\dots,x_n,y_1,\dots,y_n$ we can assume that the remaining $x_i$ and $y_i$ are disjoint and when grouped by size the groups strictly interlace:
\[a\ \leq\ x_{i_1} = \dots = x_{i_k}\ <\ y_{j_1} = \dots = y_{j_l}\ <\ \dots\ < x_{i_p} = \dots = x_{i_q}\ <\ y_{j_r} = \dots = y_{j_s}=b.\]
Each such group in $(a,b)$ has an even number of members.
\end{thm}
\begin{proof}
By \Cref{exm:algETsystem} we have that $\cF$ on $[a,b]$ is an ET-system. We then apply \Cref{cor:nonnegabDecomp} similar to the proof of \Cref{thm:algPosSatzab}.
\end{proof}

The signs of $c_*$ and $c^*$ are determined by $x_1$ and $y_1$ and their multiplicity.
If $x_1 = \dots = x_k < x_{k+1} \leq \dots \leq x_n$ then $\sign\, c_* = (-1)^k$.
The same holds for $c^*$ from the $y_i$.

\begin{cor}
If $\alpha_0 = 0$ in \Cref{thm:algNichtNegab} then \Cref{thm:algNichtNegab} also holds for $\alpha_0=0$.
\end{cor}

\begin{exm}
Let $\alpha\in (0,\infty)$ and let $\cF = \{1,x^{\alpha}\}$ on $[0,1]$. Then we have $1 = 1_* + 1^*$ with $1_* = x^\alpha$ and $1^* = 1 - x^\alpha$.\exmsymbol
\end{exm}

\subsection{Sparse Positivstellensätze and Nichtnegativstellensätze on $[0,\infty)$}

In \Cref{sec:posabGeneral} we have seen the general Positivstellen- and Nichtnegativitätsstellensätze for T-systems and then applied these to the algebraic cases on $[a,b]$.
We now show how the results from \Cref{sec:posabGeneral} on $[a,b]$ can be transferred to $[0,\infty)$.

\begin{thm}[{\cite{karlin63}} or e.g.\ {\cite[p.\ 168, Thm.\ 8.1]{karlinStuddenTSystemsBook}}]\label{thm:karlin0infty}
Let $n\in\nset$ and $\cF = \{f_i\}_{i=0}^n$ be a continuous T-system of order $n$ on $[0,\infty)$ such that
\begin{enumerate}[(a)]
\item there exists a $C>0$ such that $f_n(x)>0$ for all $x\geq C$,

\item $\displaystyle\lim_{x\to\infty} \frac{f_i(x)}{f_n(x)} = 0$ for all $i=0,\dots, n-1$, and

\item $\{f_i\}_{i=0}^{n-1}$ is a continuous T-system on $[0,\infty)$.
\end{enumerate}
Then for any $f = \sum_{i=0}^n a_i f_i\in\lin\cF$ with $f>0$ on $[0,\infty)$ and $a_n>0$ there exists a unique representation
\[f = f_* + f^*\]
with $f_*,f^*\in\lin\cF$ with $f_*,f^*\geq 0$ on $[0,\infty)$ such that the following hold:
\begin{enumerate}[(i)]
\item If $n = 2m$ the polynomials $f_*$ and $f^*$ each possess $m$ distinct zeros $\{x_i\}_{i=1}^m$ and $\{y_i\}_{i=0}^{m-1}$ satisfying
\[0 = y_0 < x_1 < y_1 < \dots < y_{m-1} < x_m <\infty.\]
All zeros except $y_0$ are double zeros.

\item If $n = 2m+1$ the polynomials $f_*$ and $f^*$ each possess the zeros $\{x_i\}_{i=1}^{m+1}$ and $\{y_i\}_{i=1}^m$ satisfying
\[0 = x_1 < y_1 < x_2 < \dots < y_m < x_{m+1} < \infty.\]
All zeros except $x_1$ are double zeros.

\item The coefficient of $f_n$ in $f_*$ is equal to $a_n$.
\end{enumerate}
\end{thm}
\begin{proof}
By (a) there exists a function $w\in C([0,\infty))$ such that $w>0$ on $[0,\infty)$ and $\lim_{x\to\infty} \frac{f_n(x)}{w(x)} = 1$.
By (b) we define
\[v_i(x) := \begin{cases}
\frac{f_i(x)}{w(x)} & \text{if}\ x\in [0,\infty),\\
\delta_{i,n} & \text{if}\ x=\infty
\end{cases}\]
for all $i=0,1,\dots,n$.
Then by (c) $\{v_i\}_{i=0}^n$ is a T-system on $[0,\infty]$ by \Cref{exm:scaling}.
With $t(x) := \tan (\pi x/2)$ we define
$g_i(x) := v_i\circ t$
for all $i=0,1,\dots,n$.
Hence, $\cG = \{g_i\}_{i=0}^n$ is a T-system on $[0,1]$.
We now apply \Cref{cor:posabDecomp} to $\cG$. Set $g := (\frac{f}{w})\circ t$.

\underline{(i):} Let $n = 2m$. Then by \Cref{cor:posabDecomp} there exits points
\[0=y_0 < x_1 < y_1 < \dots < x_m < y_m = 1\]
and unique functions $g_*$ and $g^*$ such that $g = g_* + g^*$, $g_*,g^*\geq 0$ on $[0,1]$, $x_1,\dots,x_m$ are the zeros of $g_*$, and $y_0,\dots,y_m$ are the zeros of $g^*$.
Then $f_* := (g_*\circ t^{-1})\cdot w$ and $f^* := (g^*\circ t^{-1})\cdot w$ are the unique components in the decomposition $f = f_* + f^*$.

\underline{(ii):} Similar to (i).

\underline{(iii):} From (i) (and (ii) in a similar way) we have $g_i(1) = 0$ for $i=0,\dots,n-1$ and $g_n(1) = 1$.
Hence, we get with $g^*(y_m=1) = 0$ that $g_n$ is not contained in $g^*$, i.e., $g_*$ has the only $g_n$ contribution because $\cG$ is linearly independent.
This is inherited by $f_*$ and $f^*$ which proves (iii).
\end{proof}

If $\cF$ is an ET-system then the $f_*$ and $f^*$ can be written down explicitly.

\begin{cor}\label{cor:posrepr0infty}
If in \Cref{thm:karlin0infty} we have additionally that $\cF$ is an ET-system on $[0,\infty)$ then the unique $f_*$ and $f^*$ are given
\begin{enumerate}[(i)]
\item for $n= 2m$ by
\[ f_*(x) = a_{2m}\cdot \det \begin{pmatrix}
f_0 & f_1 & f_2 & \dots & f_{2m-1} & f_{2m}\\
x & (x_1 & x_1) & \dots & (x_m & x_m)
\end{pmatrix}\]
and
\[f^*(x) = -c\cdot \det \begin{pmatrix}
f_0 & f_1 & f_2 & f_3 & \dots & f_{2m-2} & f_{2m-1}\\
x & y_0 & (y_1 & y_1) & \dots & (y_{m-1} & y_{m-1})
\end{pmatrix},\]

\item and for $n= 2m+1$ by
\[f_*(x) = -a_{2m+1}\cdot \det \begin{pmatrix}
f_0 & f_1 & f_2 & f_3 & \dots & f_{2m} & f_{2m+1}\\
x & x_1 & (x_2 & x_2) & \dots & (x_{m+1} & x_{m+1})
\end{pmatrix}\]
and
\[f^*(x) = c\cdot \det \begin{pmatrix}
f_0 & f_1 & f_2 & \dots & f_{2m-1} & f_{2m}\\
x & (y_1 & y_1) & \dots & (y_m & y_m)
\end{pmatrix}\]
\end{enumerate}
for some $c > 0$.
\end{cor}
\begin{proof}
Combine \Cref{thm:karlin0infty} with \Cref{rem:doublezeros}.
\end{proof}

If we now plug \Cref{exm:algPoly} into \Cref{thm:karlin0infty} we get the following.

\begin{thm}[Sparse Algebraic Positivstellensatz on {$[0,\infty)$}]\label{thm:algPosSatz0infty}
Let $n\in\nset$, $\alpha_0,\dots,\alpha_n\in [0,\infty)$ be real numbers with $\alpha_0 = 0 < \alpha_1 < \dots < \alpha_n$, and let $\cF = \{x^{\alpha_i}\}_{i=0}^n$ on $[0,\infty)$.
Then for any $f = \sum_{i=0}^n a_i f_i\in\lin\cF$ with $f>0$ on $[0,\infty)$ and $a_n>0$ there exists a unique decomposition
\[f = f_* + f^*\]
with $f_*,f^*\in\lin\cF$ and $f_*,f^*\geq 0$ on $[0,\infty)$ such that the following hold:
\begin{enumerate}[(i)]
\item If $n = 2m$ then the polynomials $f_*$ and $f^*$ each possess $m$ distinct zeros $\{x_i\}_{i=1}^m$ and $\{y_i\}_{i=0}^{m-1}$ satisfying
\[0 = y_0 < x_1 < y_1 < \dots < y_{m-1} < x_m < \infty.\]
The polynomials $f_*$ and $f^*$ are given by
\[f_*(x) = a_{2m}\cdot\det\begin{pmatrix}
1 & x^{\alpha_1} & x^{\alpha_2} & \dots & x^{\alpha_{2m-1}} & x^{\alpha_{2m}}\\
x & (x_1 & x_1) & \dots & (x_m & x_m)
\end{pmatrix}\]
and
\begin{align*}
f^*(x) 
%
&= c\cdot \det\begin{pmatrix}
x^{\alpha_1} & x^{\alpha_2} & x^{\alpha_3} &\dots& x^{\alpha_{2m-2}} & x^{\alpha_{2m-1}}\\
x & (y_1 & y_1) & \dots & (y_{m-1} & y_{m-1})
\end{pmatrix}
\end{align*}
for some $c>0$.

\item If $n=2m+1$ then $f_*$ and $f^*$ have zeros $\{x_i\}_{i=1}^{m+1}$ and $\{y_i\}_{i=1}^m$ respectively which satisfy
\[0 = x_1 < y_1 < x_2 < \dots < y_m < x_{m+1}<\infty.\]
The polynomials $f_*$ and $f^*$ are given by
\begin{align*}
f_*(x) 
&= a_{2m+1}\cdot \det\begin{pmatrix}
x^{\alpha_1} & x^{\alpha_2} & x^{\alpha_3} & \dots & x^{\alpha_{2m}} & x^{\alpha_{2m+1}}\\
x & (x_2 & x_2) & \dots & (x_{m+1} & x_{m+1})
\end{pmatrix}
\end{align*}
and
\[f^*(x) = c\cdot\det\begin{pmatrix}
1 & x^{\alpha_1} & x^{\alpha_2} & \dots & x^{\alpha_{2m-1}} & x^{\alpha_{2m}}\\
x & (y_1 & y_1) & \dots & (y_m & y_m)
\end{pmatrix}\]
for some $c>0$.
\end{enumerate}
\end{thm}
\begin{proof}
We have that $\cF$ clearly fulfill condition (a) and (b) of \Cref{thm:karlin0infty} and by \Cref{exm:algPoly} we known that $\cF$ on $[0,\infty)$ is also a T-system, i.e., (c) in \Cref{thm:karlin0infty} is fulfilled.
We can therefore apply \Cref{thm:karlin0infty}.

\underline{(i) $n=2m$:} By \Cref{thm:karlin0infty}(i) the unique $f_*$ and $f^*$ each possess $m$ distinct zeros $\{x_i\}_{i=1}^m$ and $\{y_i\}_{i=0}^{m-1}$ with $0\leq y_0 < x_1 < \dots < y_{m-1} < x_m < \infty$.
Since $x_1,\dots,x_m\in (0,\infty)$ and $\cF$ on $[x_1/2,\infty)$ is an ET-system we immediately get the determinantal representation of $f_*$ by \Cref{cor:posrepr0infty} (combine \Cref{thm:karlin0infty} with \Cref{rem:doublezeros}).
For $f^*$ we have $y_0=0$ and by \Cref{exm:nonETalg} this is no ET-system.
Hence, we prove the representation of $f^*$ by hand.

Let $\varepsilon>0$ be such that $0=y_0 < y_1 < y_1+\varepsilon < \dots < y_{m-1} < y_{m-1}+\varepsilon$ holds. Then
\begin{align*}
g_\varepsilon(x) &= -\varepsilon^{-m+1}\cdot \det \begin{pmatrix}
1 & x^{\alpha_1} & x^{\alpha_2} & x^{\alpha_3} & \dots & x^{\alpha_{2m-2}} & x^{\alpha_{2m-1}}\\
x & 0 & y_1 & y_1+\varepsilon & \dots & y_{m-1} & y_{m-1}+\varepsilon
\end{pmatrix}\\
&= -\varepsilon^{-m+1}\cdot \det\begin{pmatrix}
1 & x^{\alpha_1} & x^{\alpha_2} & \dots & x^{\alpha_{2m-1}}\\
1 & 0 & 0 & \dots & 0\\
1 & y_1^{\alpha_1} & y_1^{\alpha_2} & \dots & y_1^{\alpha_{2m-1}}\\
\vdots & \vdots & \vdots & & \vdots\\
1 & (y_{m-1}+\varepsilon)^{\alpha_1} & (y_{m-1}+\varepsilon)^{\alpha_2} & \dots & (y_{m-1}+\varepsilon)^{\alpha_{2m-1}}
\end{pmatrix}
\intertext{expand by the second row}
&= \varepsilon^{-m+1}\cdot\det\begin{pmatrix}
x^{\alpha_1} & x^{\alpha_2} & \dots & x^{\alpha_{2m-1}}\\
y_1^{\alpha_1} & y_1^{\alpha_2} & \dots & y_1^{\alpha_{2m-1}}\\
\vdots & \vdots & & \vdots\\
(y_{m-1}+\varepsilon)^{\alpha_1} & (y_{m-1}+\varepsilon)^{\alpha_2} & \dots & (y_{m-1}+\varepsilon)^{\alpha_{2m-1}}
\end{pmatrix}\\
&= \varepsilon^{-m+1}\cdot\det\begin{pmatrix}
x^{\alpha_1} & x^{\alpha_2} & \dots & x^{\alpha_{2m-2}} & x^{\alpha_{2m-1}}\\
x & y_1 & y_1+\varepsilon & \dots & y_{m-1} & y_{m-1}+\varepsilon
\end{pmatrix}
\end{align*}
is non-negative on $[0,y_1]$ and every $[y_i+\varepsilon,y_{i+1}]$.
Now $y_0=0$ is removed and all $y_i,y_i+\varepsilon>0$.
Hence, we can work on $[y_1/2,\infty)$ where $\{x^{\alpha_i}\}_{i=1}^{2m}$ is an ET-system and we can go to the limit $\varepsilon\to 0$ as in \Cref{rem:doublezeros}.
Then \Cref{cor:posrepr0infty} proves the representation of $f^*$.

\underline{(ii) $n=2m+1$:} Similar to the case (i) with $n=2m$.
\end{proof}

The previous result was reproved in \cite{scheider23}.
Additionally, since the authors of \cite{scheider23} were not aware of \cite{karlin63,karlinStuddenTSystemsBook} their statement is much weaker and the proof is unnecessary long and complicated.
In \cite{scheider23} several other results are reproved which already appeared in \cite{karlinStuddenTSystemsBook}.

It is left to the reader to use \Cref{cor:nonnegabDecomp} to gain the corresponding sparse Nichtnegativestellensatz on $[0,\infty)$ for general T-systems and for $\{x^{\alpha_i}\}_{i=0}^n$ with $0= \alpha_0 < \alpha_1 < \dots < \alpha_n$ real numbers.
The proofs follow the same line of thoughts as the proof of \Cref{thm:algNichtNegab}.
If all $\alpha_i\in\nset_0$ then we can express the $f_*$ and $f^*$ in \Cref{thm:algPosSatz0infty} also with Schur polynomials, see (\ref{eq:schurPolyRepr}) in \Cref{exm:algETsystem}.

We have seen that Boas already investigated the sparse Stieltjes moment problem \cite{boas39}.
However, since Boas did not have access to \Cref{thm:karlin0infty} by Karlin and therefore \Cref{thm:algPosSatz0infty} the description was complicated and incomplete.
We get the following complete and simple description.
To our knowledge this result did not appear somewhere else.

\begin{thm}[Sparse Stieltjes Moment Problem]\label{thm:sparseStieltjesMP}
Let $\{\alpha_i\}_{i\in\nset_0}\subset [0,\infty)$ such that $\alpha_0 = 0 < \alpha_1 < \alpha_2 < \dots$ and let $\cF = \{x^{\alpha_i}\}_{i\in\nset_0}$.
Then the following are equivalent:
\begin{enumerate}[(i)]
\item $L:\lin\cF\to\rset$ is a $[0,\infty)$-moment functional.

\item $L(p)\geq 0$ for all $p\in\lin\cF$ with $p\geq 0$.

\item $L(p)\geq 0$ for all $p\in\lin\cF$ with $p>0$.

\item $L(p)\geq 0$ for all
\[p(x) = \begin{cases}
\det\begin{pmatrix}
1 & x^{\alpha_1} & x^{\alpha_2} & \dots & x^{\alpha_{2m-1}} & x^{\alpha_{2m}}\\
x & (x_1 & x_1) & \dots & (x_m & x_m)
\end{pmatrix},\\
\det\begin{pmatrix}
x^{\alpha_1} & x^{\alpha_2} & x^{\alpha_3} &\dots& x^{\alpha_{2m-2}} & x^{\alpha_{2m-1}}\\
x & (x_1 & x_1) & \dots & (x_{m-1} & x_{m-1})
\end{pmatrix},\\
\det\begin{pmatrix}
x^{\alpha_1} & x^{\alpha_2} & x^{\alpha_3} & \dots & x^{\alpha_{2m}} & x^{\alpha_{2m+1}}\\
x & (x_2 & x_2) & \dots & (x_{m+1} & x_{m+1})
\end{pmatrix},\ \text{and}\\
\det\begin{pmatrix}
1 & x^{\alpha_1} & x^{\alpha_2} & \dots & x^{\alpha_{2m-1}} & x^{\alpha_{2m}}\\
x & (x_1 & x_1) & \dots & (x_m & x_m)
\end{pmatrix}
\end{cases}\]
for all $m\in\nset_0$ and $0<x_1<\dots<x_m$.
\end{enumerate}
\end{thm}
\begin{proof}
The implications ``(i) $\Rightarrow$ (ii) $\Leftrightarrow$ (iii)'' are clear and ``(iii) $\Leftrightarrow$ (iv)'' is \Cref{thm:algPosSatz0infty}.
It is therefore sufficient to prove ``(ii) $\Leftarrow$ (i)''.

We have $\lin\cF = (\lin\cF)_+ - (\lin\cF)_+$, we have $1 = x^{\alpha_0}\in\lin\cF$, and for any $g = \sum_{i=0}^m a_i\cdot x^{\alpha_i} \in (\lin\cF)_+$ we have $\lim_{x\to\infty} \frac{g(x)}{x^{\alpha_{m+1}}} = 0$, i.e., there exists an $f\in (\lin\cF)_+$ which dominates $g$.
Hence, $\lin\cF$ is an adapted space and the assertion follows from \cite[Thm.\ 1.8]{schmudMomentBook}.
\end{proof}

Note, in the previous result we did needed $0=\alpha_0 < \alpha_1 < \alpha_2 < \dots$ but we did not need $\alpha_i\to\infty$.
Hence, \Cref{thm:sparseStieltjesMP} also includes the case $\sup_{i\in\nset_0}\alpha_i < \infty$.

\subsection{Sparse Positivstellensätze and Nichtnegativstellensätze on $\rset$}

\begin{thm}[{\cite{karlin63}} or e.g.\ {\cite[p.\ 198, Thm.\ 8.1]{karlinStuddenTSystemsBook}}]\label{thm:karlinR}
Let $m\in\nset_0$ and $\cF = \{f_i\}_{i=0}^{2m}$ be a continuous T-system of order $2m$ on $\rset$ such that
\begin{enumerate}[(a)]
\item there exists a $C>0$ such that $f_{2m}(x)>0$ for all $x\in (-\infty,-C]\cup [C,\infty)$,

\item $\displaystyle\lim_{|x|\to\infty} \frac{f_i(x)}{f_{2m}(x)} = 0$ for all $i=0,\dots,2m-1$, and

\item $\{f_i\}_{i=0}^{2m-1}$ is a continuous T-system of order $2m-1$ on $\rset$.
\end{enumerate}
Let $f = \sum_{i=0}^{2m} a_i f_i$ be such that $f>0$ on $\rset$ and $a_{2m}>0$.
Then there exists a unique representation
\[f = f_* + f^*\]
with $f_*,f^*\in\lin\cF$ and $f_*,f^*\geq 0$ on $\rset$ such that
\begin{enumerate}[(i)]
\item the coefficient of $f_{2m}$ in $f_*$ is $a_{2m}$, and

\item $f_*$ and $f^*$ are non-negative polynomials having zeros $\{x_i\}_{i=1}^m$ and $\{y_i\}_{i=1}^{m-1}$ with
\[-\infty < x_1 < y_1 < x_2 < \dots < y_{m-1} < x_m < \infty.\]
\end{enumerate}
\end{thm}
\begin{proof}
Adapt the proof of \Cref{thm:karlin0infty} such that both interval ends of $[a,b]$ are mapped to $-\infty$ and $+\infty$, respectively.
\end{proof}

We have already seen how from \Cref{thm:karlinab} and \Cref{cor:posabDecomp} we gained \Cref{thm:algPosSatzab} (sparse algebraic Positivstellensatz on $[a,b]$), \Cref{thm:algNichtNegab} (sparse algebraic Nichtnegativstellensatz on $[a,b]$), and Theorems \ref{thm:sparseHausd} and \ref{thm:generalSparseHausd} (sparse Hausdorff moment problems).
We have seen how from \Cref{thm:karlinab} and \Cref{cor:posabDecomp} we gained \Cref{thm:karlin0infty} (sparse Positivstellensatz for T-systems on $[0,\infty)$), \Cref{thm:algPosSatz0infty} (sparse algebraic Positivstellensatz on $[0,\infty)$), and \Cref{thm:sparseStieltjesMP} (sparse Stieltjes moment problem).
We will therefore not repeat this procedure for the case $K = \rset$ from \Cref{thm:karlinR} but summarize the procedure in the following ``cooking receipt''.

\begin{rem}[A General Cooking Receipt]\label{rem:cooking}
We have the following general \emph{cooking receipt} for generating sparse Positivstellensätze, Nichtnegativstellensätze, and to generate and solve sparse moment problems:
\begin{enumerate}[\quad (A)]
\item Use \Cref{thm:karlinab} or \Cref{cor:posabDecomp}, or extend these to sets $K= [a,b]\cup [c,d],\dots$ (for extensions see e.g.\ \cite{karlinStuddenTSystemsBook} and later literature on T-systems we did not discuss here).

\item Prove that your family $\cF = \{f_i\}_{i\in I}$ is a T-system (or even an ET-system).

\item Plug $\cF$ into (A) to get the sparse Positivstellensatz or sparse Nichtnegativestellensatz on $K$.

\item Show that $\lin\cF$ is an adapted space.

\item Combine (C) and (D) to a sparse moment problem (use \cite[Thm.\ 1.8]{schmudMomentBook} for an efficient proof).
\end{enumerate}
With this cooking receipt a large class of (sparse) moment problems, Nichtnegativestellensätze, and Positivstellensätze can be generated, solved, and efficiently proved.
We think this makes it very useful for applications and further theoretical investigations.
\exmsymbol
\end{rem}

\section{Summary}
\label{sec:summary}

In this work we review and deal with univariate sparse moment problems, Positivstellensätzen, and Nichtnegativestellensätzen.
We look at earlier results and then move to the theory of T-systems.
In the center are the works of Karlin \cite{karlin63} and Karlin and Studden \cite{karlinStuddenTSystemsBook}.

From \Cref{thm:karlinab} and \Cref{cor:posabDecomp} on $[a,b]$ we deduce a complete description of all strictly positive sparse algebraic polynomials in \Cref{thm:algPosSatzab}.
We also give the sparse algebraic Nichtnegativestellensatz on $[a,b]$ in \Cref{thm:algNichtNegab}. With these results we completely solve the sparse Hausdorff moment problem in \Cref{thm:sparseTruncHausd}, \Cref{thm:sparseHausd}, and for the most general form in \Cref{thm:generalSparseHausd}.

Following the extension by Karlin and Studden of \Cref{thm:karlinab} and \Cref{cor:posabDecomp} from $[a,b]$ to $[0,\infty)$ we formulate the corresponding sparse algebraic Positivstellensatz on $[0,\infty)$ in \Cref{thm:algPosSatz0infty}.
Only the sparse algebraic Positivstellensatz on $[0,\infty)$ is given since it already solves the sparse Stieltjes moment problem in \Cref{thm:sparseStieltjesMP}.
The sparse algebraic Nichtnegativstellensatz on $[0,\infty)$ can easily be derived like the sparse Nichtnegativestellensatz on $[a,b]$ in \Cref{thm:algNichtNegab}.

We also give the general T-system Positivstellensatz on $\rset$ by Karlin in \Cref{thm:karlinR}.
We give a general ``cooking receipt'' how other results in \cite{karlinStuddenTSystemsBook} (and later literature) can be used to generate additional sparse algebraic Positivstellensätze, Nichtnegativestellensätze, or to formulate and solve sparse moment problems.

In this treatment we see the high value of the results in \cite{karlinStuddenTSystemsBook} which are rarely used today.
Especially the analytic treatment of the algebraic questions seems unusual at first.
However, we hope we convinced the reader that this approach has at least in the univariate case (\Cref{thm:msc}) great value to gain sparse Positivstellensätze, sparse Nichtnegativstellensätze, and solutions to sparse moment problems.



\section*{Funding}

The author and this project are supported by the Deutsche Forschungs\-gemein\-schaft DFG with the grant DI-2780/2-1 and his research fellowship at the Zukunfs\-kolleg of the University of Konstanz, funded as part of the Excellence Strategy of the German Federal and State Government.


\providecommand{\bysame}{\leavevmode\hbox to3em{\hrulefill}\thinspace}
\providecommand{\MR}{\relax\ifhmode\unskip\space\fi MR }
\providecommand{\MRhref}[2]{%
  \href{http://www.ams.org/mathscinet-getitem?mr=#1}{#2}
}
\providecommand{\href}[2]{#2}

\end{document}